\documentclass[12pt]{amsart}

\usepackage{amsfonts,amsthm,amsmath,amssymb,amscd,mathrsfs,tikz}
\usepackage{graphics}
\usepackage{indentfirst}
\usepackage{cite}
\usepackage[title]{appendix}
\usepackage{latexsym}
\usepackage[dvips]{epsfig}
\usepackage{color}
\usepackage{enumerate}
\newtheorem{theorem}{Theorem}[section]

\newtheorem{lemma}[theorem]{Lemma}
\newtheorem{pro}[theorem]{Proposition}

\renewcommand{\div}{{\rm div }}

\newcommand{\bt}{\begin{theorem}}
	\newcommand{\bl}{\begin{lemma}}
		\newcommand{\el}{\end{lemma}}
	\newcommand{\et}{\end{theorem}}

\newcommand{\bn}{\begin{eqnarray}}
	\newcommand{\en}{\end{eqnarray}}
\newcommand{\bnn}{\begin{eqnarray*}}
	\newcommand{\enn}{\end{eqnarray*}}

\newcommand{\ba}{\begin{aligned}}
	\newcommand{\ea}{\end{aligned}}
\newcommand{\be}{\begin{equation}}
	\newcommand{\ee}{\end{equation}}

\newcommand{\whBo}{\widehat{\mathscr{B}_{1}}}
\newcommand{\witBo}{\widetilde{\mathscr{B}_{1}}}
\newcommand{\sfwitBo}{\widehat{\mathscr{B}_{\frac{3}{4}}}}

\newcommand{\bBV}{\boldsymbol{V}}
\newcommand{\Bn}{{\boldsymbol{n}}}
\newcommand{\Bt}{{\boldsymbol{\tau}}}
\newcommand{\Bu}{{\boldsymbol{u}}}
\newcommand{\Be}{{\boldsymbol{e}}}

\newcommand{\D}{{\boldsymbol{D}}}

\newcommand{\Ps}{\mathbf{\Psi}}

\newcommand{\supp}{\text{supp}}

\newcommand{\BPsi}{{\boldsymbol{\Psi}}}

\newcommand{\Bx}{{\boldsymbol{x}}}
\newcommand{\Bf}{{\boldsymbol{f}}}
\newcommand{\By}{{\boldsymbol{y}}}
\newcommand{\Bo}{{\boldsymbol{\omega}}}

\newcommand{\OR}{{\mathscr{O}}_{R}}
\newcommand{\DR}{{\mathcal{D}}_{R}}

\begin{document}
	
	\title[Liouville-type theorems]
	{Liouville-type theorems for steady  Navier-Stokes system under helical symmetry or Navier boundary conditions}

	\author{Jingwen Han}
	\address{School of Mathematical Sciences, Shanghai Jiao Tong University, 800 Dongchuan Road, Shanghai, China}
	\email{hjw126666@sjtu.edu.cn}

	\author{Yun Wang}
	\address{School of Mathematical Sciences, Center for dynamical systems and differential equations, Soochow University, Suzhou, China}
	\email{ywang3@suda.edu.cn}

	\author{Chunjing Xie}
	\address{School of mathematical Sciences, Institute of Natural Sciences,
		Ministry of Education Key Laboratory of Scientific and Engineering Computing,
		and CMA-Shanghai, Shanghai Jiao Tong University, 800 Dongchuan Road, Shanghai, China}
	\email{cjxie@sjtu.edu.cn}

	\begin{abstract}
		In this paper,  the Liouville-type theorems for the steady Navier-Stokes system are investigated. First, we prove that any bounded smooth helically symmetric solution in $\mathbb{R}^3$  must be a constant vector. Second,  for steady Navier-Stokes system in a slab supplemented with Navier boundary conditions, we  prove that any bounded smooth solution must be zero if either the swirl or radial velocity is axisymmetric, or $ru^{r}$ decays to zero as $r$ tends to infinity. Finally, when the velocity is not big in $L^{\infty}$-space, the general three-dimensional steady Navier-Stokes flow in a slab with the Navier boundary conditions must be a Poiseuille type flow. The key idea of the proof  is to establish Saint-Venant  type estimates that characterize the growth of Dirichlet integral of nontrivial solutions.
	\end{abstract}


	
	\maketitle

	\section{Introduction and Main Results}
	Classical Liouville theorem asserts that  bounded harmonic functions in whole spaces are constants.  The generalization of this kind of classification results for PDEs is called Liouville-type theorem nowadays, which has many applications for PDEs, such as analyzing rigidity, hypothetical singularity and asymptotic behavior of solutions. In this paper, we are interested in the Liouville-type theorem for solutions to the steady incompressible Navier-Stokes system,
	\begin{equation}\label{eqsteadyns}
		\left\{ \ba
		& -\Delta \Bu + (\Bu \cdot \nabla )\Bu + \nabla P  = 0, \ \ \ \ \
		&\mbox{in}\ \Omega, \\
		& \nabla \cdot \Bu =0,  \ \ \ \ \ &\mbox{in}\ \Omega,  \\
		\ea \right.
	\end{equation}
	where the unknown function $\Bu=(u^1,u^2,u^3)$  is the velocity field,  $P$ is the pressure, $\Omega$ is the domain.
	
	A  weak solution $\Bu$ to \eqref{eqsteadyns}  is called D-solution if it satisfies
	\begin{equation}\label{eqDs}
		\int_{\Omega}|\nabla\Bu|^{2}\, d\Bx< +\infty.
	\end{equation}
	The existence of D-solution in both bounded  and exterior domains was first proved by  Leray \cite{LEJMPA33}. A longstanding open problem is whether the solution $\Bu(\Bx)$ in the whole space equals to zero if it vanishes at infinity \cite[X.9 Remark X.9.4]{GAGP11}.  The problem in $\mathbb{R}^2$ was solved by Gilbarg and   Weinberger \cite{GWASP78}, where
	D-solutions  were
	proved to be constants. This significant  result was generalized in  \cite{KPRJMFM15,CWDCDS16} for axisymmetric D-solution in $\mathbb{R}^3$. However, whether three-dimensional D-solutions in the whole space are constants is still a widely open problem, although there are many studies on this problem in the last several decades. In particular, the important progress has been made when some extra integrability or decay conditions on velocity $\Bu$ or vorticity $\Bo$ were prescribed.  It was  proved in \cite[Theorem X. 9.5]{GAGP11} that the solution must be zero if  $\Bu\in L^{\frac{9}{2}}(\mathbb{R}^{3})$.  For more references in this direction, one may refer to \cite{CWJDE16,CGPLAP21,YX20JMAA,KTWJFA17,JMFM13bfz,CCMP14,CYJMAA13,SN16,WJDE19,NZ19} and references therein.
	
	On the other hand, a natural and important problem is to classify $L^{\infty}$-bounded solutions of Navier-Stokes system, which has many applications in characterization of singularity of solutions and asymptotic behavior of solutions at far fields. For instance, an important progress in \cite{CSYTIMRN08,CSYTCPDE09,KNSS09} suggests that every bounded axisymmetric steady solution with type I bound for Navier-Stokes system must vanish. In fact, straightforward computations show that bounded steady solutions of Stokes system in $\mathbb{R}^{n}$ must be constants.
	As far as Navier-Stokes system is concerned,
	every bounded two-dimensional  and three-dimensional axisymmetric without swirl  flows in the whole space must be  constant vectors (cf. \cite{KNSS09}). A further important progress  in \cite{LRZMA22} indicated that the bounded steady solution of axisymmetric Navier-Stokes system in $\mathbb{R}^{2}\times \mathbb{T}$ must be trivial  if $ru^{\theta}$ is bounded.
	
	In this paper, we are concerned with the Liouville theorem for helically symmetric flows. Helical symmetry is invariant under a one-dimensional subgroup $\mathcal{G}_{\kappa}$ of rigid motions generated by a simultaneous rotation around a fixed axis and translation along the same  symmetric axis. Namely, the subgroup $\mathcal{G}_{\kappa}$ is a one-parameter group of isometries of $\mathbb{R}^{3}$,
	\begin{equation}\label{eqA2}
		\mathcal{G}_{\kappa}=
		\left\{
		S_{\rho}:\mathbb{R}^{3}
		\rightarrow
		\mathbb{R}^{3} \big|  \rho\in\mathbb{R}
		\right\},
	\end{equation}
	here the transformation $S_{\rho}$ (its graph is shown in Figure  \ref{Fig1-1}) is defined by
	\begin{equation}\label{eqA3}
		S_{\rho}(\Bx)=
		R_{\rho}(\Bx)
		+\left(
		\begin{array}{c}
			0 \\
			0 \\
			\kappa \rho
		\end{array}
		\right),
	\end{equation}
	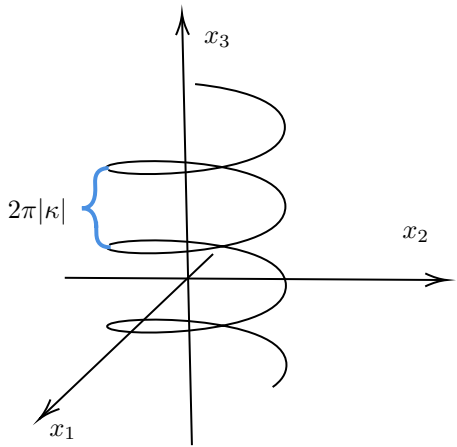
\begin{figure}[h]
		\begin{center}
			\centering

			\tikzset{every picture/.style={line width=0.75pt}} 
			
			\begin{tikzpicture}[x=0.75pt,y=0.75pt,yscale=-1,xscale=1]
				
				\draw   (305.92,56.83) .. controls (328.51,59.18) and (351.12,65.53) .. (351.21,78.53) .. controls (351.38,104.53) and (261.09,105.13) .. (261.05,99.13) .. controls (261.01,93.13) and (351.3,92.53) .. (351.48,118.53) .. controls (351.65,144.53) and (261.35,145.13) .. (261.31,139.13) .. controls (261.27,133.13) and (351.57,132.53) .. (351.74,158.53) .. controls (351.91,184.53) and (261.62,185.13) .. (261.58,179.13) .. controls (261.54,173.13) and (351.83,172.53) .. (352.01,198.53) .. controls (352.04,202.98) and (349.41,206.69) .. (345.04,209.75) ;
				\draw    (240.2,154.6) -- (430.93,155.72) ;
				\draw [shift={(432.93,155.73)}, rotate = 180.34] [color={rgb, 255:red, 0; green, 0; blue, 0 }  ][line width=0.75]    (10.93,-3.29) .. controls (6.95,-1.4) and (3.31,-0.3) .. (0,0) .. controls (3.31,0.3) and (6.95,1.4) .. (10.93,3.29)   ;
				\draw    (315,142.6) -- (228.78,224.56) ;
				\draw [shift={(227.33,225.93)}, rotate = 316.45] [color={rgb, 255:red, 0; green, 0; blue, 0 }  ][line width=0.75]    (10.93,-3.29) .. controls (6.95,-1.4) and (3.31,-0.3) .. (0,0) .. controls (3.31,0.3) and (6.95,1.4) .. (10.93,3.29)   ;
				\draw    (304.33,239.08) -- (299.38,23.08) ;
				\draw [shift={(299.33,21.08)}, rotate = 88.69] [color={rgb, 255:red, 0; green, 0; blue, 0 }  ][line width=0.75]    (10.93,-3.29) .. controls (6.95,-1.4) and (3.31,-0.3) .. (0,0) .. controls (3.31,0.3) and (6.95,1.4) .. (10.93,3.29)   ;
				\draw  [color={rgb, 255:red, 74; green, 144; blue, 226 }  ,draw opacity=1 ][fill={rgb, 255:red, 255; green, 255; blue, 255 }  ,fill opacity=1 ][line width=1.5]  (262.33,99.19) .. controls (257.67,99.31) and (255.4,101.7) .. (255.51,106.36) -- (255.59,109.37) .. controls (255.75,116.04) and (253.5,119.43) .. (248.84,119.54) .. controls (253.5,119.43) and (255.91,122.7) .. (256.08,129.36)(256.01,126.36) -- (256.16,132.37) .. controls (256.27,137.04) and (258.66,139.31) .. (263.33,139.19) ;
				
				\draw (231,227) node [anchor=north west][inner sep=0.75pt]  [font=\footnotesize] [align=left] {$\displaystyle x_{1}$};
				\draw (409,127.4) node [anchor=north west][inner sep=0.75pt]  [font=\footnotesize] [align=left] {$\displaystyle x_{2}$};
				\draw (309,29) node [anchor=north west][inner sep=0.75pt]  [font=\footnotesize] [align=left] {$\displaystyle x_{3}$};
				\draw (210,113) node [anchor=north west][inner sep=0.75pt]  [font=\footnotesize] [align=left] {$\displaystyle 2\pi |\kappa |$};

			\end{tikzpicture}

		\end{center}
		\caption{the transfomation $S_{\rho}$}
		\label{Fig1-1}

	\end{figure}
	where $R_{\rho}$ is the rotation matrix by an angle $\rho$ around the $x_{3}$-axis, i.e.,
	\[
	R_{\rho}=\left(
	\begin{array}{ccc}
		\cos\rho  &      \sin\rho & 0 \\
		-\sin\rho & 	\cos\rho & 0 \\
		0         &            0 & 1
	\end{array}
	\right),
	\]
	and the nonzero constant $\kappa$ denotes the translation along the $x_{3}$-axis, which we call the step or pitch. A smooth function $h(\Bx)$ and vector field $\Bu(\Bx)$ are helically symmetric if $h\left(S_{\rho}(\Bx)\right)=h(\Bx)$ and $\Bu(S_{\rho}(\Bx))= R_{\rho}\Bu(\Bx)$,  for all $\rho\in \mathbb{R}$, respectively. Hence any helically symmetric flow is  periodic in $x_{3}$-direction, with   a period $2\pi|\kappa|$. Throughout this paper we will assume $\kappa=\frac{1}{2\pi}$ for simplicity.
	
	Helically symmetric flows have been widely studied in various aspects.  The global  regularity of helically symmetric flows  was first studied in \cite{MTLARMA90}. In \cite{LMNLT14,BLNNT13},  the planar limit and  stability of the  helically symmetric flows were achieved. Recently, the existence of helical invariant weak solutions in a helical domain for the steady Navier-Stokes system with helically symmetric external force was obtained in \cite{KLWSIAM22}. In a periodic or infinite pipe subject to the total Navier-slip boundary conditions,  it was proved  that the smooth solution must be a helically symmetric  flow under some conditions \cite{ALPY}. For more references related to helically symmetric flows, one may refer to \cite{BLLN15,ET09,JLN17,LN17,JMCLN18PHYD} and references therein.
	
	Our first main result is a Liouville-type theorem for helically symmetric steady flows.
	\begin{theorem}\label{th:01}
		Assume that  $\Bu$ is a   bounded smooth helically symmetric solution to the  Navier-Stokes system \eqref{eqsteadyns} in $\mathbb{R}^{3}$. Then $\Bu$  must be a constant vector of the form $\left(0, 0, C\right)$.
	\end{theorem}
	As we mentioned  before,  helically symmetric flows can be regarded as flows in a slab with periodic boundary conditions. In the slab domain,  the solvability and asymptotic behavior for Navier-Stokes system  with Dirichlet boundary conditions were studied in \cite{NAKPJMFM1999,KPsbm02,KPSNAA10}.   For the Liouville-type theorem, when  supplemented with  no-slip boundary conditions, it was  first proved in \cite{CPZJFA20,CPZZARMA20} that the D-solution must be zero. The axisymmetric D-solution with periodic  boundary conditions was also proved to be trivial. It should be mentioned that this result was improved in \cite{PJMP21} under the assumption that only the  swirl  velocity  is axisymmetric. The same result for solutions with full slip boundary conditions was also proved in \cite{PJMP21}. In addition, the condition \eqref{eqDs} was relaxed  in  \cite{TT21}  and  some Liouville-type results were established under local integrability condition of $\Bu$. Very recently,    the Liouville-type theorems for bounded solutions of  the Navier-Stokes system \eqref{eqsteadyns} with either no-slip boundary conditions or periodic boundary conditions were  studied in \cite{aBGWX}, where it was proved that any bounded axisymmetric solution is trivial. More generally, the Liouville theorem also holds if either $u^r$ or $u^\theta$ is axisymmetric. In particular, when the flows are  supplemented with no-slip boundary conditions on the slab boundary, they also proved that if the velocity is not big in $L^{\infty}$ space, the general three-dimensional solutions must be Poiseuille flows. The method in \cite{aBGWX} was also used to study the Liouville-type theorem for  Taylor-Couette flows (cf. \cite{KTWarXi23}).
	
	In this paper, we consider the problem  in the slab $\Omega=\mathbb{R}^2\times(0,1)$, equipped with  Navier boundary conditions  (cf. \cite{NMARSI27})
	\begin{equation}\label{eqNavierbou}
		\Bu\cdot\Bn=0, \ \ \ \left(\Bn\cdot\D(\Bu)+\alpha\Bu\right)
		\cdot \Bt=0,
		\quad  \text{at} \ x_3=0\, \,\text{and} \,\,1,
	\end{equation}
	where $\D(\Bu)$  is the strain tensor defined by
	$$ \D(\Bu)_{i,j}=
	(\partial_{x_j}u^{i}+\partial_{x_i}u^{j})/2,
	$$
	and $\alpha \geq  0$ is the friction coefficient which measures the tendency of a fluid over the boundary, $\Bt$ and $\Bn$ are the unit tangent and outer normal vector on the boundary $\partial \Omega$, respectively. If $\alpha=0$,
	the boundary conditions \eqref{eqNavierbou} are  called the full (total) slip boundary conditions. If $\alpha \to +\infty$, the boundary conditions \eqref{eqNavierbou} formally reduce to the classical no-slip boundary conditions.
	In the standard Cartesian coordinates framework $\left\{\Be_1,\, \Be_2,\, \Be_3\right\}$, let $\Bu=u^{1} \Be_1+u^{2}\Be_2
	+u^{3}\Be_3$, the  Navier boundary conditions \eqref{eqNavierbou} on the slab boundaries become
	\begin{equation}\label{bounNaviersl1}
		u^{3}=0,   \ \ \
		\mp\dfrac{\partial_{x_{3}}u^{1}}{2}+\alpha u^{1}=0, \ \ \
		\mp\dfrac{\partial_{x_{3}}u^{2}}{2}+\alpha u^{2}=0,  
		\quad \text{at} \ x_{3}=0\, \,\text{and} \,\,1.
	\end{equation}
	
	The well-posedness of the non-stationary Navier-Stokes system supplemented with  Navier boundary conditions were studied in \cite{MRLARMA12,HBV07Math,JPKSIAM06, LMNPSIAM05,XXCPAM07}. For the stationary Navier-Stokes system, in the case the friction coefficient $\alpha=0$, i.e., the full slip boundary case,   the existence and uniqueness of very-weak solution in appropriate Banach spaces for the three-dimensional Navier-Stokes system in the flat boundary were studied in \cite{BDCDS10}. Later on, the existence and uniqueness of weak and strong solutions for stationary Stokes and Navier-Stokes system in Sobolev spaces were investigated in \cite{TACGJDE21}, when the positive friction function $\alpha$ admits minimal regularity. For more references on the well-posedness of the stationary Stokes or Navier-Stokes system  prescribed by Navier boundary conditions, one may refer to \cite{ARJDE14,HBV04ADE,ARARMA11,SWX} and references therein.
	
	In terms of the  cylindrical coordinates $(r,\theta,z)$, which are defined as follows
	\begin{equation}\label{eq6}
		\Bx=(x_{1}, x_{2}, x_{3})
		=(r\cos\theta, r\sin\theta,z),
	\end{equation}
	one can rewrite the  velocity $\Bu$ as
	\begin{equation}\label{eq7}
		\Bu=u^{r}(r,\theta,z)\Be_{r}
		+u^{\theta}(r,\theta,z)\Be_{\theta}
		+u^{z}(r,\theta,z)\Be_{z},
	\end{equation}
	where $u^{r}, u^{\theta}, u^{z}$ are called radial, swirl and axial velocity, respectively, with
	\[
	\Be_{r}=(\cos\theta, \sin\theta, 0), \ \ \ \Be_{\theta}=(-\sin\theta, \cos\theta, 0) \ \ \ \text{and} \ \ \
	\Be_{z}=(0, 0, 1).
	\]	
	With the aid of cylindrical coordinates, the Navier-Stokes system \eqref{eqsteadyns} and Navier boundary conditions \eqref{eqNavierbou}  can be written as
	\begin{equation}\label{eqA8}
		\left\{
		\begin{aligned}
			&\left(u^r \partial_r + \frac{u^\theta}{r}  \partial_\theta + u^z \partial_z\right)u^r - \frac{(u^\theta)^2}{r} + \frac{2}{r^2} \partial_\theta u^\theta + \partial_r P = \left(\Delta_{r,\theta, z} - \frac{1}{r^2} \right) u^r,\\
			&\left(u^r \partial_r + \frac{u^\theta }{r}  \partial_\theta + u^z \partial_z\right)u^\theta +
			\frac{u^\theta u^r}{r} - \frac{2}{r^2} \partial_\theta u^r + \frac{1}{r} \partial_\theta P=  \left(\Delta_{r,\theta, z} - \frac{1}{r^2} \right) u^\theta,\\
			&\left(u^r \partial_r + \frac{u^\theta}{r} \partial_\theta + u^z \partial_z\right)u^z + \partial_z P =\Delta_{r,\theta, z}u^z,\\
			&\partial_r u^r +\frac{1}{r}\partial_\theta u^\theta+\partial_z u^z +\frac{u^r}{r}=0
		\end{aligned}
		\right.
	\end{equation}
	and
	\begin{equation}\label{eqA10}
		u^{z}=0,   \ \ \
		\mp\dfrac{\partial_{z}u^{r}}{2}+\alpha u^{r}=0, \ \ \
		\mp\dfrac{\partial_{z}u^{\theta}}{2}+\alpha u^{\theta}=0,  \quad 
		\text{at} \ z=0\, \, \text{and} \,\,1,
	\end{equation}
	respectively,
	where
	\[
	\Delta_{r,\theta,z}
	=\partial^{2}_{r}+\dfrac{1}{r}\partial_{r}
	+\dfrac{1}{r^{2}}\partial^{2}_{\theta}
	+\partial^{2}_{z}.
	\]
	
	If $u^{r}\left(u^{\theta}, u^{z}, \text{respectively}\right)$ does not depend on $\theta$, we say $u^{r}\left(u^{\theta}, u^{z},\text{respectively}\right)$ is axisymmetric. Similarly, if the  vector $\left(u^{r}, u^{\theta}, u^{z}\right)$ does not depend on $\theta$, we say that the flow is  axisymmetric.
	Our second main result is the following Liouville-type theorem for  flows  in a slab with Navier boundary conditions \eqref{eqNavierbou}.
	\begin{theorem}\label{th:02}
		Let $\Bu$ be a bounded smooth   solution to the  Navier-Stokes system \eqref{eqsteadyns} in  the slab $\Omega=\mathbb{R}^{2}\times (0,1)$ supplemented with Navier boundary conditions \eqref{eqNavierbou}. The following statements hold.
		\begin{enumerate}
			\item[(i)] $\Bu\equiv 0$ if one of the following conditions holds:
			\begin{enumerate}
				\item[(a)] $u^{\theta}$  or $u^{r}$ is axisymmetric,
				
				\item[(b)] $ru^r$ uniformly  converges to $0$, as $r\rightarrow +\infty$;
			\end{enumerate}
		\end{enumerate}
		\begin{enumerate}
			\item[(ii)]if  $\alpha=0$ and $\|\Bu\|_{L^{\infty}(\Omega)}<\pi$, $\Bu$ must be a constant vector of the form $\left(C_{1}, C_{2}, 0\right)$;
			\item[(iii)]if  $\alpha>0$ and $	\|\Bu\|_{L^{\infty}(\Omega)}<
			\min\left\{\dfrac{\alpha}{2}, \, \dfrac{1}{4}\right\}$, $\Bu$ must be a Poiseuille type flow of the  form  $\left(C_{3}\left(-2\alpha x^{2}_{3}+2\alpha x_{3}+1\right), C_{4}\left(-2\alpha x^{2}_{3}+2\alpha x_{3}+1\right), 0\right)$.
		\end{enumerate}
	\end{theorem}	
	In fact, we can even prove the following Liouville-type theorem for axisymmetric solutions when the velocity fields have sublinear growth at far field.
	\begin{theorem}\label{th:03}
		If the friction coefficient $\alpha>0$ and the solution $\Bu$ is axisymmetric, then $\Bu\equiv 0$ provided that
		\begin{equation} \label{growth-coro1}
			\lim_{R\rightarrow + \infty} R^{-1}  \sup_{z \in [0, 1]} |\Bu (R, z)|        =0.
		\end{equation}
	\end{theorem}
	Now we outline the key point of the proof.   The first key observation of this paper is that periodic property and helical identities \eqref{dhelicdef} together give some nice estimates. Inspired by the work \cite{aBGWX}, making  use of  the structure of the steady Navier-Stokes system,  Bogovskii map and helical identities, we prove the Saint-Venant type estimate  for the Dirichlet integral of $\Bu$ over a finite subdomain. The Saint-Venant's principle which dates back to  \cite{JKKTARMA66,RATARMA65} was originally used to study the solutions for elastic equations.   This idea  was generalized in  \cite{OYASN77} to  investigate the  second order elliptic equations, where the uniqueness and existence  of solutions in unbounded domains were obtained. Later on, it was applied in  \cite{LSZNSL80} to study the the famous Leray problem, i.e., the well-posedness of the steady Stokes and Navier-Stokes system in domains with unbounded boundaries. Furthermore. when the flows in a slab satisfy the full slip boundary conditions, it can be regarded as a flow periodic in one direction after suitable extension. When the flows in a slab satisfy the Navier boundary conditions with positive friction coefficient, then one can combine the Poincar\'e inequality for $u^{z}$ together with divergence free property of velocity field to show  Poincar\'e type inequality  for the whole velocity field.
	
	The rest of this paper is organized as follows.  In Section \ref{Sec2}, we introduce some notations and collect some elementary lemmas which are used in this paper. The Liouville-type theorem for helically symmetric flows in $\mathbb{R}^{3}$ is presented in Section \ref{Sec3}. In Section \ref{Sec4}, we  study the flows in a slab  with full slip boundary conditions via periodic extension. In Section \ref{Sec5}, the Liouville-type theorems for  the flows with Navier boundary conditions with $\alpha>0$ are investigated. More precisely,  we consider the bounded flows with only  swirl axisymmetry  or  radial axisymmetry in Section \ref{Sec51}. Section \ref{Sec52} is devoted to the proof of Liouville-type theorem for  general  bounded flows with $ru^{r}$ decaying to zero.
	In Section \ref{Sec54}, we consider the flows  whose $L^\infty$ norm are not big. In Section \ref{Sec6},  we  prove the Liouville-type theorem for  axisymmetric solutions with sublinear growth. The  proof for the regularity of  bounded solutions to Navier-Stokes system with Navier boundary conditions is presented in detail in the appendix.

	\section{Preliminaries}\label{Sec2}
	In this section, we give some preliminaries. First, we introduce the following notations. Assume that $\Omega$ is a bounded domain, define
	\[
	L^p_0(\Omega)=\left\{g: \ \ g\in L^p(\Omega), \ \ \int_\Omega g \, d\Bx =0 \right\}.
	\]
	For any $R\geq 2$, denote
	$D_R = (R-1, R)\times (0, 1)$, $\mathcal{D}_R =(R-1, R)\times(0, 2\pi)\times(0, 1)$,  $\Omega_R = B_R \times (0, 1)$ and $\OR = (B_R \setminus \overline{B_{R-1}}) \times (0, 1)$, where $B_R= \{(x_1,x_2)\in \mathbb{R}^2: x_1^2+x_2^2 <R^2\}$.
	For any $\Bx\in \mathbb{R}^3$, define $\mathscr{B}_r(\Bx)=\{\By\in \mathbb{R}^3: |\By-\Bx|<r\}$.
	In the rest of the paper, the  cut-off function $\varphi_R(r)$ is defined as following
	\be \label{cut-off}
	\varphi_R(r) = \left\{ \ba
	&1,\ \ \ \ \ \ \ \ \ \ r < R-1, \\
	&R-r,\ \ \ \ R-1 \leq r \leq R, \\
	&0, \ \ \ \ \ \ \ \ \ \ r > R.
	\ea  \right.
	\ee
	
	Subsequently, we introduce the Bogovskii map, which gives a solution to the divergence equation. The proof is due to  Bogovskii \cite{B79DANS}, see also  \cite[Section III.3]{GAGP11}  and  \cite[Section 2.8]{TT18}.
	\begin{lemma}(\hspace{1sp}\cite[Lemma 2.1]{aBGWX}, \cite[Theorem 1.2]{BN22})\label{Bogovskii}
		Let $\Omega$ be a bounded Lipschitz domain in $\mathbb{R}^n$ with $n\geq 2 $.  For any $q\in (1, \infty)$, there is a linear map $\boldsymbol{\Phi}$ that maps a scalar function $g\in L^q_0(\Omega)$ to a vector field $\bBV = \boldsymbol{\Phi} g \in W_0^{1, q}(\Omega; \mathbb{R}^n)$ satisfying
		\be \nonumber
		{\rm div}~\bBV = g \ \text{in}\ \Omega \quad \text{and} \quad \|\bBV\|_{W^{1, q}(\Omega)} \leq C (\Omega, q) \|g\|_{L^q(\Omega)}.
		\ee
		Moreover, if $\Omega$ is a bounded $C^{k, 1}$ domain in $\mathbb{R}^n$ and $g\in W^{k, q}(\Omega)$, $k\in \mathbb{N}$,  it holds that
		\be \nonumber
		\|\bBV\|_{W^{k+1, q}(\Omega)} \leq C (\Omega, k, q) \|g\|_{W^{k, q}(\Omega)}.
		\ee
		In particular,
		\begin{enumerate}
			\item    For any $g \in L^2_0(D_R)$,
			the  vector valued function $\bBV = \boldsymbol{\Phi} g \in H_0^1(D_R; \mathbb{R}^2)$ satisfies
			\be \nonumber
			\partial_r V^r + \partial_z V^z =g \ \  \mbox{in}\,\, D_R
			\quad
			\text{and}
			\quad
			\|\tilde{\nabla } \bBV\|_{L^2(D_R)}
			\leq C \|g\|_{L^2(D_R)},
			\ee
			where $\tilde{\nabla } = (\partial_r, \ \partial_z ) $ and $C$ is a constant independent of $R$.
			
			\item  For any $g \in L^2(\mathcal{D}_R)$,
			the   vector valued function $\bBV = \boldsymbol{\Phi} g \in H_0^1( \mathcal{D}_R; \mathbb{R}^3)$ satisfies
			\be \nonumber
			\partial_r V^r + \partial_\theta V^\theta +  \partial_z V^z =g \ \   \mbox{in}\,\, \mathcal{D}_R
			\quad
			\text{and}
			\quad
			\|\bar{\nabla } \bBV\|_{L^2(\mathcal{D}_R)}
			\leq C \|g\|_{L^2( \mathcal{D}_R)},
			\ee
			where $\bar{\nabla} = (\partial_r, \ \partial_\theta, \ \partial_z ) $ and $C$ is a constant independent of $R$.
		\end{enumerate}
	\end{lemma}
	
	Next, some  properties of the helical flows are collected in the following lemma.
	\begin{lemma}\label{Leheliclfl}
		If	a continuously differentiable vector field $\Bu=u^{r}(r, \theta, z) \Be_{r}+
		u^{\theta}(r,\theta,z)\Be_{\theta}+
		u^{z}(r,\theta,z)\Be_{z}$ is helically symmetric, then there exists a constant $\kappa \in \mathbb{R}$  such that
		\begin{equation}\label{dhelicdef}
			\kappa\partial_{z}u^{r}
			=\partial_{\theta}u^{r},
			\ \ \	\kappa\partial_{z}u^{\theta}=\partial_{\theta}u^{\theta},
			\ \ \	\kappa\partial_{z}u^{z}=\partial_{\theta}u^{z}.
		\end{equation}
	\end{lemma}
	\begin{proof}
		As shown in \cite[p. 846]{LMNLT14} and \cite[p. 3976]{KLWSIAM22}, a continuous vector field $\Bu$ is helically symmetric if and only if there exist $\left(v^r, v^\theta, v^z\right)$ and  a constant $\kappa \in \mathbb{R}$ such that
		\begin{equation}\label{helidef}
			u^{r}(r,\theta,z)= 	v^{r}(r,\kappa\theta+z),\ \ \
			u^{\theta}(r,\theta,z)= 	v^{\theta}(r,\kappa\theta+z),\ \ \
			u^{z}(r,\theta,z)= 	v^{z}(r,\kappa\theta+z).
		\end{equation}
		Differentiating \eqref{helidef} with respect to $\theta$ and $z$ variables  gives \eqref{dhelicdef}.  The proof of  Lemma \ref{Leheliclfl} is completed.
	\end{proof}
	For a solution of Navier-Stokes system in a slab with periodic boundary conditions, if the velocity field is $L^{\infty}$-bounded, then its gradient must also be $L^{\infty}$-bounded.
	The proof can be found in \cite[Lemma 2.3]{aBGWX},  so we omit it here.
	\begin{lemma}(\hspace{1sp}\cite[Lemma 2.3]{aBGWX})\label{le:pebu}
		Let $\Bu$ be a bounded smooth solution to the  Navier-Stokes system \eqref{eqsteadyns} in $\mathbb{R}^{2}\times \mathbb{T}$. Then $\nabla\Bu$, $\nabla^{2}\Bu$, and $\nabla P$ are uniformly bounded.
	\end{lemma}
	
	For bounded solutions of the  Navier-Stokes system in a slab supplemented with Navier boundary conditions if the friction coefficient $\alpha>0$, one can also show that the gradient of the velocity field is   $L^{\infty}$-bounded. The detailed proof is given in the appendix.
	\begin{lemma}\label{Le:NAvbouns}
		Let $\Bu$ be a bounded smooth solution to the  Navier-Stokes system \eqref{eqsteadyns} in a slab $\Omega=\mathbb{R}^{2}\times (0,1)$ supplemented with Navier boundary conditions \eqref{eqNavierbou}, where $\alpha>0$. Then $\nabla\Bu$ is bounded.
	\end{lemma}	
	
	In order to apply Saint-Vernant principle, the following estimate for the differential inequalities are frequently used in this paper. The proof can be found in \cite[Lemma 2.2]{aBGWX}.
	\begin{lemma}(\hspace{1sp}\cite[Lemma 2.2]{aBGWX})\label{le:differineq}
		Let $y(t)$ be a nondecreasing nonnegative function  and $t_0>1$ be a fixed constant. Suppose that $y(t)$ is not identically zero.
		
		(a)  If $y(t)$  satisfies
		\begin{equation}\label{ineqlemma2-1}
			y(t) \leq C_1 y^{\prime}(t) + C_2 \left[ y^{\prime}(t) \right]^{\frac32} \ \ \ \text{for any}\,\, t\geq t_0,
		\end{equation}
		then
		\begin{equation} \label{qwelemma2-2}
			\varliminf_{t \rightarrow + \infty} t^{-3} y(t) > 0.
		\end{equation}
		
		(b) If $y(t)$ satisfies
		\begin{equation}\label{eqinlemma2-3}
			y(t) \leq C_3 y^{\prime}(t) + C_4  t^{-\frac12} \left[ y^{\prime}(t) \right]^{\frac32} \ \ \ \text{for any}\,\, t\geq t_0,
		\end{equation}
		then
		\begin{equation}\label{eqinlemma2-4}
			\varliminf_{t \rightarrow + \infty} t^{-4} y(t)>0.
		\end{equation}
	\end{lemma}
	
	The following Korn's inequality (cf. \cite[Theorem 2.1]{CPGCAA10}) and the Ne\v{c}as inequality (cf. \cite[Theorem IV.1.1]{BooBOYFNSb} and \cite[Theorem 1]{LN21MMAS}) are useful to study the solutions of Navier-Stokes system with Navier boundary conditions.
	\begin{lemma}\label{KNINEQ26}
		(a) (Korn's inequality)  Let $\Omega_{1}$ be a  domain in $\mathbb{R}^3$,  then there exists a constant $C=C(\Omega_{1})$ such that 
		\begin{equation}\label{kornin1}
			\|\Bu\|_{H^{1}(\Omega_{1})} \leq C\left(\|\Bu\|_{L^{2}\left(\Omega_{1}\right)}+\|\D(\Bu)\|_{\mathbb{L}^{2}\left(\Omega_{1}\right)}\right), \quad \text{for  all} \,\, \Bu\in H^{1}(\Omega_{1}).
		\end{equation}
		
		(b) (Ne\v{c}as inequality)  Let $\Omega_{2}$ be a bounded Lipschitz domain in $\mathbb{R}^3$,  then there exists a constant $C=C(\Omega_{2})$ such that 
		\begin{equation}\label{necas1in}
			\|f\|_{L^{2}(\Omega_{2})} \leq C\left(\|f\|_{H^{-1}\left(\Omega_{2}\right)}+\|\nabla f\|_{H^{-1}\left(\Omega_{2}\right)}\right), \quad \text{for  all} \,\, f\in L^{2}(\Omega_{2}).
		\end{equation}
	\end{lemma}
	
	An important property for the flows for which the velocity $\Bu$ is bounded and periodic in one direction is that the pressure is also periodic. The proof can be found in \cite[Lemma 5.1]{aBGWX}.
	\begin{lemma} (\hspace{1sp}\cite[Lemma 5.1]{aBGWX})\label{Le:peridiPR}
		Let $\Bu$ be a bounded smooth solution to the  Navier-Stokes system \eqref{eqsteadyns} in $\mathbb{R}^{2}\times \mathbb{T}$. The pressure $P$ is also a periodic function with respect to $z$.
	\end{lemma}
	\section{Helically Symmetric Flows}\label{Sec3}
	The section is  devoted to the proof of Theorem \ref{th:01}. Note that the helically symmetric flow must be periodic along the axial direction. Hence the analysis on Liouville-type theorem for steady Navier-Stokes flows in a slab with periodic boundary conditions \cite[Section 5]{aBGWX} can be applied.  The key ingredient of this paper  is that we make full use of  the helical  identities \eqref{dhelicdef} here. In this section, without loss of generality, we assume that $\kappa = \frac{1}{2\pi}$.
	
	The following lemma shows that the helically symmetric flow must be a constant vector if the associated Dirichlet integral is finite.
	\begin{lemma}\label{Le:32}
		Let $\Bu$ be a bounded smooth helically symmetric solution to the  Navier-Stokes system \eqref{eqsteadyns} in $\mathbb{R}^{3}$, and  $\Bu$  has a finite Dirichlet integral in the slab, i.e.,
		\begin{equation}\label{21Dintass}
			\int_{\mathbb{R}^{2}\times
				(0,\, 1)}|\nabla\Bu|^{2} \,d\Bx<+\infty.
		\end{equation}
		Then $\Bu$ must be a constant vector of the form  $(0, 0, C)$.
	\end{lemma}
	\begin{proof}[Proof]
		The proof is divided into two steps.
		
		\emph{Step 1.} \emph {Set up.}	 Since $\Bu$ is a bounded smooth helically symmetric solution to \eqref{eqsteadyns} in $\mathbb{R}^{3}$,  it follows from Lemma \ref{Le:peridiPR} that the pressure $P$ is also a periodic function with respect to $z$.  Multiplying the first equation in \eqref{eqsteadyns} by $\varphi_{R}(r)\Bu$ and integrating by parts over  the  slab $\Omega=\mathbb{R}^{2}\times (0,1)$, one obtains	
		\begin{equation}\label{muinteeq}
			\int_{\Omega}|\nabla\Bu|^{2}\varphi_{R} \, d\Bx
			=-\int_{\Omega}\nabla\varphi_{R}
			\cdot \nabla\Bu\cdot \Bu \,d\Bx + \int_{\Omega}\frac{1}{2}|\Bu|^{2}\Bu\cdot \nabla\varphi_{R} \,d\Bx + \int_\Omega P \Bu \cdot \nabla \varphi_R \,d\Bx.
		\end{equation}
		The straightforward computations give
		\begin{equation}\nonumber
			\int_\Omega P \Bu \cdot \nabla \varphi_R \,d\Bx
			= -\int_0^1 \int_0^{2\pi} \int_{R-1}^R P u^{r}r \, dr d\theta dz.
		\end{equation}
		Using the divergence free and helically symmetric  properties \eqref{dhelicdef}, for all  $0\leq r<+\infty$,  yields
		\begin{equation}\label{eqA116}
			\begin{split}
				\partial_{r}\int_{0}^{1}ru^{r}dz=-\int_{0}^{1}\partial_{\theta}u^{\theta}+\partial_{z}(ru^{z})\,dz=-\int_{0}^{1}\dfrac{1}{2\pi}\partial_{z}u^{\theta}+\partial_{z}(ru^{z})\, dz =0.
			\end{split}
		\end{equation}
		This implies
		\begin{equation}\label{eqA117}
			\int_{0}^{1}ru^{r}\, dz=0
			\quad \text{and}\quad \int_{0}^{1} \int_{R-1}^{R}ru^{r}\, drdz=0.
		\end{equation}
		It follows from \eqref{eqA117} and Lemma \ref{Bogovskii} that for every fixed $\theta\in [0,2\pi]$,
		there exists a vector valued function $\Ps_{R,\theta}(r,z)\in H^{1}_{0}(D_{R}; \mathbb{R}^{2})$ satisfying
		\begin{equation}\label{eqA118}
			\partial_{r}\Psi_{R,\theta}^{r}
			+\partial_{z}\Psi_{R,\theta}^{z}
			=ru^{r},
		\end{equation}
		together with the estimate
		\begin{equation}\label{eqA119}
			\|(\partial_{r}, \partial_z) \Ps_{R,\theta}\|_{L^{2}(D_{R})} \leq C\|ru^{r}\|_{L^{2}(D_{R})},
		\end{equation}
		where $C$ is independent of $\theta$ and $R$.
		Owing to \eqref{eqA117}, the Poincar\'{e} inequality
		\begin{equation}\label{eqA122}
			\|u^{r}\|_{L^{2}(\OR)}\leq	C\|\partial_{z}u^{r}\|_{L^{2}(\OR)} \end{equation}
		holds.
		This, together with \eqref{eqA119}, gives
		\begin{equation}\label{eqA123}
			\|(\partial_{r}, \partial_z)\Ps_{R,\theta}\|_{L^{2}(\DR)} \leq C\|ru^{r}\|_{L^{2}(\DR)}\leq CR^{\frac{1}{2}}\|u^{r}\|_{L^{2}(\OR)}\leq CR^{\frac{1}{2}}\|\nabla\Bu\|_{L^{2}(\OR)}.
		\end{equation}
		Note that the Bogovskii map is a linear map  \cite[Section III.3]{GAGP11}. Hence there is a universal constant $C>0$ such that
		\begin{equation}\label{eqA124}
			\|(\partial_{\theta}\partial_{r}, \partial_{\theta}\partial_{z}) \Ps_{R,\theta}\|_{L^{2}(\DR)} \leq C\|r\partial_{\theta}
			u^{r}\|_{L^{2}(\DR)}\leq CR^{\frac{1}{2}}\|\partial_{z} u^{r}\|_{L^{2}(\OR)},
		\end{equation}
		where  the last inequality is due to \eqref{dhelicdef}.
		Furthermore, it follows from Lemma \ref{Le:peridiPR} and \eqref{eqA118} that one has
		\begin{equation}\label{preeqes}
			\begin{split}
				\int_{\Omega}P\Bu\cdot \nabla \varphi_{R} \,d\Bx
				=&-\int_{0}^{1}
				\int_{0}^{2\pi}
				\int_{R-1}^{R}P\cdot ru^{r}
				\, drd\theta dz\\
				=&-\int_{0}^{1}
				\int_{0}^{2\pi}
				\int_{R-1}^{R}P(\partial_{r}\Psi_{R,\theta}^{r}+\partial_{z}\Psi_{R,\theta}^{z}) \, drd\theta dz\\
				=& \int_{0}^{1}\int_{0}^{2\pi}
				\int_{R-1}^{R}(\partial_{r}P\Psi_{R,\theta}^{r}+ \partial_{z}P\Psi_{R,\theta}^{z})\, drd\theta dz.
			\end{split}
		\end{equation}
		According to \eqref{eqA8}, the gradient  of the pressure $(\partial_{r}P, \partial_{z}P)$ satisfies
		\begin{equation}\label{eqPGe}
			\left\{
			\begin{aligned}
				&\left(u^{r}\partial_{r}+\dfrac{u^{\theta}}{r}\partial_{\theta}
				+u^{z}\partial_{z}\right) u^{r}-\dfrac{(u^{\theta})^{2}}{r}
				+\dfrac{2}{r^{2}}\partial_{\theta} u^{\theta}+\partial_{r}P
				=\left(\Delta_{r,\theta,z}-\dfrac{1}{r^{2}}\right)u^{r},\\
				&\left(u^{r}\partial_{r}
				+\dfrac{u^{\theta}}{r}\partial_{\theta}+u^{z}\partial_{z}\right) u^{z}+\partial_{z}P
				=\Delta_{r,\theta,z}u^{z},
			\end{aligned}
			\right.
		\end{equation}
		where
		\[
		\Delta_{r,\theta,z}=\partial^{2}_{r}
		+\dfrac{1}{r}\partial_{r}
		+\dfrac{1}{r^{2}}\partial^{2}_{\theta}
		+\partial^{2}_{z}.
		\]
		By virtue of  \eqref{eqPGe} and integration by parts, one has
		\begin{equation}\label{eqArP}
			\begin{split}
				& \int_{0}^{1}\int_{0}^{2\pi}\int_{R-1}^{R}\partial_{r}P\Psi^{r}_{R,\theta}\, drd\theta dz\\
				=& -\int_{0}^{1}\int_{0}^{2\pi}\int_{R-1}^{R}\left(\partial_{r}u^{r}\partial_{r}\Psi^{r}_{R,\theta}+\partial_{z}u^{r}\partial_{z}\Psi^{r}_{R,\theta}
				+\dfrac{1}{r^2}\partial_{\theta}u^{r}\partial_{\theta}\Psi^{r}_{R,\theta}\right)\, drd\theta dz \\
				&+\int_{0}^{1}\int_{0}^{2\pi}\int_{R-1}^{R}\left[\left(\dfrac{1}{r}\partial_{r}-\dfrac{1}{r^2}\right)u^{r}-\dfrac{2}{r^2}\partial_{\theta}u^{\theta}\right]\Psi^{r}_{R,\theta}\, drd\theta dz  \\
				&-\int_{0}^{1}\int_{0}^{2\pi}\int_{R-1}^{R}\left[\left(u^{r}\partial_{r}+\dfrac{u^{\theta}}{r}\partial_{\theta}+u^{z}\partial_{z}\right)u^{r}-\dfrac{(u^{\theta})^2}{r}\right]\Psi_{R,\theta}^{r}\, drd\theta dz
			\end{split}
		\end{equation}
		and
		\begin{equation}\label{eqAzP}
			\begin{split}
				&\int_{0}^{1}\int_{0}^{2\pi}\int_{R-1}^{R}\partial_{z}P\Psi^{z}_{R,\theta}\, drd\theta dz\\
				=&	-\int_{0}^{1}\int_{0}^{2\pi}\int_{R-1}^{R}\left(\partial_{r}u^{z}\partial_{r}\Psi_{R,\theta}^{z}+\partial_{z}u^{z}\partial_{z}\Psi_{R,\theta}^{z}+\dfrac{1}{r^2}\partial_{\theta}u^{z}\partial_{\theta}\Psi_{R,\theta}^{z}\right)\, drd{\theta}dz\\
				& -\int_{0}^{1}\int_{0}^{2\pi}\int_{R-1}^{R}\left[\left(u^{r}\partial_{r}+
				\dfrac{u^{\theta}}{r}\partial_{\theta}+u^{z}\partial_{z}-\dfrac{1}{r}\partial_{r}\right)u^{z}\right]\Psi_{R,\theta}^{z}\, drd\theta dz.
			\end{split}
		\end{equation}
		\emph{Step 2.} \emph {Saint-Venant type estimate.} Making use of helical identities \eqref{dhelicdef}, Poincar\'e inequality \eqref{eqA122},  the estimates   \eqref{eqA123}-\eqref{eqA124}  and $\|u^{r}\|_{L^{2}(\OR)}\leq CR^{\frac{1}{2}}\|\Bu\|_{L^{\infty}(\OR)}$, one has
		\begin{equation}\label{eqA126}
			\begin{split}
				&\left| \int_{0}^{1}\int_{0}^{2\pi}\int_{R-1}^{R}\left(\partial_{r}u^{r}\partial_{r}\Psi_{R,\theta}^{r}+\partial_{z}u^{r}\partial_{z}\Psi_{R,\theta}^{r}\right)\, drd\theta dz\right|\\
				\leq&      CR^{-\frac{1}{2}}\|\nabla\Bu\|_{L^{2}(\OR)}\cdot   R^{\frac{1}{2}}\|u^{r}\|_{L^{2}(\OR)}\leq  CR^{\frac{1}{2}} \|\nabla\Bu\|_{L^{2}(\OR)}
			\end{split}
		\end{equation}
		and
		\begin{equation}\label{eqA127}
			\begin{split}
				\left| \int_{0}^{1}\int_{0}^{2\pi}
				\int_{R-1}^{R}\dfrac{1}{r^{2}}\partial_{\theta} u^{r}\partial_{\theta} \Psi_{R,\theta}^{r}\, drd\theta dz\right|
				\leq& CR^{-\frac{5}{2}}\|\partial_{z}u^{r}\|_{L^{2}(\OR)}\cdot R^{\frac{1}{2}}\|\nabla\Bu\|_{L^{2}(\OR)}\\
				\leq& CR^{-2}\|\nabla\Bu\|^{2}_{L^{2}(\OR)}\leq C\|\nabla\Bu\|_{L^{2}(\OR)},
			\end{split}
		\end{equation}	
		where  the last inequality is obtained with the aid of the  assumption \eqref{21Dintass}. Furthermore, one has	
		\begin{equation}\label{eqA128}
			\begin{split}
				&\left|\int_{0}^{1}\int_{0}^{2\pi}\int_{R-1}^{R}\left[\left(\dfrac{1}{r}\partial_{r}-\dfrac{1}{r^2}\right)u^{r}
				-\dfrac{2}{r^{2}}\partial_{\theta}u^{\theta}\right]\Psi_{R,\theta}^{r}\, drd\theta dz \right|	\\
				\leq& CR^{-1}R^{-\frac{1}{2}}\| \nabla\Bu\|_{L^{2}(\OR)}
				\cdot R^{\frac{1}{2}} \|u^{r}\|_{L^{2}(\OR)}
				\leq CR^{-\frac{1}{2}}\| \nabla\Bu\|_{L^{2}(\OR)}
			\end{split}
		\end{equation}
		and
		\begin{equation}\label{eqA129}
			\begin{split}			   &\left|\int_{0}^{1}\int_{0}^{2\pi}
				\int_{R-1}^{R}\left[\left(u^{r}\partial_{r}+\dfrac{u^{\theta}}{r}\partial_{\theta}+u^{z}\partial_{z}\right)u^{r}-\dfrac{(u^{\theta})^{2}}{r}\right]\Psi_{R,\theta}^{r} \,drd\theta dz \right|	\\
				\leq&
				C\|\Bu\|_{L^{\infty}(\OR)}
				\left(R^{-\frac{1}{2}}\|\nabla\Bu\|_{L^{2}(\OR)}+R^{-\frac{3}{2}}\|u^{\theta}\|_{L^{2}(\OR)}\right)\cdot R^{\frac{1}{2}}\|u^{r}\|_{L^{2}(\OR)}\\
				\leq&
				CR^{\frac{1}{2}}\| \nabla\Bu\|_{L^{2}(\OR)}.
			\end{split}
		\end{equation}
		Collecting the estimates \eqref{eqA126}-\eqref{eqA129} gives
		\begin{equation}\label{eqA130}
			\left|\int_{0}^{1}\int_{0}^{2\pi}
			\int_{R-1}^{R}\partial_{r}P\Psi^{r}_{R,\theta}\, drd\theta dz\right|
			\leq CR^{\frac{1}{2}}\|\nabla\Bu\|_{L^{2}(\OR)}.
		\end{equation}
		Similarly, it holds that
		\begin{equation}\label{eqA131}
			\left|\int_{0}^{1}\int_{0}^{2\pi}
			\int_{R-1}^{R}\partial_{z}P\Psi^{z}_{R,\theta}\, drd\theta dz\right|
			\leq CR^{\frac{1}{2}}\|\nabla\Bu\|_{L^{2}(\OR)}.
		\end{equation}
		We now estimate the first two terms on the right hand side of \eqref{muinteeq}.
		Using  H{\"o}lder inequality  yields
		\begin{equation}\label{basces1}
			\left|\int_{\Omega}\nabla\varphi_{R}\cdot \nabla\Bu\cdot\Bu  \,d\Bx \right|\leq C\|\nabla\Bu\|_{L^{2}(\OR)}\cdot R^{\frac{1}{2}}\|\Bu\|_{L^{\infty}(\OR)}\leq CR^{\frac{1}{2}}\|\nabla\Bu\|_{L^{2}(\OR)}.
		\end{equation}
		It follows from	Poincar\'{e} inequality \eqref{eqA122} that
		\begin{equation}\label{basces2}
			\left|\int_{\Omega}\frac{1}{2}\left|\Bu \right|^2\Bu\cdot \nabla \varphi_{R} \,d\Bx\right|\leq C\|\Bu\|^{2}_{L^{\infty}(\OR)} \|u^{r}\|_{L^{2}(\OR)}\cdot R^{\frac{1}{2}}\leq CR^{\frac{1}{2}}\|\nabla\Bu\|_{L^{2}(\OR)}.
		\end{equation}
		Combining  \eqref{eqA130}-\eqref{basces2}, one arrives at
		\begin{equation}\label{eqbaesin}
			\int_{\Omega}|\nabla \Bu|^{2}\varphi_{R} \,d\Bx
			\leq CR^{\frac{1}{2}}\|\nabla\Bu\|_{L^{2}(\OR)}.
		\end{equation}
		
		Let
		\begin{equation}\label{gDinsgro}
			Y(R)=
			\int_{0}^{1}\iint_{\mathbb{R}^{2}}
			|\nabla\Bu|^{2}\varphi_{R}\left(\sqrt{x_{1}^{2}+x_{2}^{2}}\right)
			\, dx_{1} dx_{2} dx_{3}.
		\end{equation}
		Straightforward computations give
		\begin{equation}\label{eqA50}
			Y(R)=\int_{0}^{1}\int_{0}^{2\pi}
			\left(\int_{0}^{R-1}|\nabla\Bu|^{2}r\,dr+\int_{R-1}^{R}|\nabla\Bu|^{2}(R-r)r\, dr\right)d\theta dz
		\end{equation}
		and
		\begin{equation}\label{eqA51}
			Y^{\prime}(R)=\int_{\OR}|\nabla\Bu|^{2} \,d\Bx.
		\end{equation}
		Hence the estimate \eqref{eqbaesin} can be written
		as
		\begin{equation}\label{eqA52}
			Y(R)\leq CR^{\frac{1}{2}} \left[Y^{\prime}(R)\right]^{\frac{1}{2}}.
		\end{equation}
		If $\nabla\Bu$ is not identically equal to zero, then $Y(R)> 0$ for $R\geq R_{0}$ with some $R_{0}>0$, and one has
		\begin{equation}\label{eqA53}
			\dfrac{1}{C^{2}R}\leq
			\left(-\dfrac{1}{Y(R)}\right)^{\prime}.
		\end{equation}
		Integrating it over $(R_{0},R)$ for large $R_{0}$, one arrives at
		\begin{equation}\label{eqA54}
			\dfrac{1}{C^{2}}\ln\dfrac{R}{R_{0}}\leq
			-\dfrac{1}{Y(R)}+\dfrac{1}{Y(R_{0})}\leq \dfrac{1}{Y(R_{0})}.
		\end{equation}
		This leads to a contradiction when $R$ is sufficiently large. Hence, $\nabla \Bu\equiv0$ and  $\Bu$ is a constant vector.	Moreover, the helical properties \eqref{helidef} give
		\begin{equation}
			\begin{split}
				\Bu=&u^{r}(r,\theta,z) \Be_{r}+
				u^{\theta}(r,\theta,z)\Be_{\theta}+
				u^{z}(r,\theta,z)\Be_{z}\\
				=&v^{r}\left(r, \frac{1}{2\pi}\theta+z\right) \Be_{r}+
				v^{\theta}\left(r, \frac{1}{2\pi}\theta+z\right)\Be_{\theta}+
				v^{z}\left(r, \frac{1}{2\pi}\theta+z\right)\Be_{z}.
			\end{split}
		\end{equation}
		Hence one has $u^{r}=u^{\theta}\equiv 0$ and $u^{z}\equiv C$. This finishes the proof of the lemma.
	\end{proof}
	Now we are ready to prove Theorem \ref{th:01}.
	\begin{proof}[Proof for  Theorem \ref{th:01}]
		Since $\Bu$ is a bounded smooth helically symmetric solution to \eqref{eqsteadyns} in $\mathbb{R}^{3}$,  it follows from Lemma \ref{Le:peridiPR} that the equality \eqref{muinteeq} still holds.
		
		The proof is almost the same as that for  Lemma \ref{Le:32}, except that
		\begin{equation}\label{eqA133}
			\begin{split}
				\left| \int_{0}^{1}\int_{0}^{2\pi}
				\int_{R-1}^{R}\dfrac{1}{r^{2}}\partial_{\theta} u^{r}\partial_{\theta} \Psi_{R,\theta}^{r}\,drd\theta dz
				\right|
				\leq& CR^{-\frac{5}{2}}\|\partial_{z}u^{r}\|_{L^{2}(\OR)}\cdot R^{\frac{1}{2}}\|\nabla\Bu\|_{L^{2}(\OR)}\\
				\leq& C\|\nabla\Bu\|^{2}_{L^{2}(\OR)}.
			\end{split}
		\end{equation}	
		The computations in the proof of Lemma \ref{Le:32} imply
		\begin{equation}\label{eqA134}
			Y(R)\leq  C_{1}Y^{\prime}(R)
			+C_{2}R^{\frac{1}{2}}
			\left[Y^{\prime}(R)\right]^{\frac{1}{2}},
		\end{equation}
		where $Y(R)$ is defined in \eqref{gDinsgro}. Hence one has
		\begin{equation}\label{eqA135}
			\left[Y^{\prime}(R)\right]^{\frac{1}{2}}
			\geq \dfrac{-C_{2}R^{\frac{1}{2}}+\sqrt{C_{2}^{2}R+4C_{1}Y(R)}}{2C_{1}}
			\geq
			\dfrac{Y(R)}{\sqrt{C_{2}^{2}R+ 4C_{1}Y(R)}}.
		\end{equation}
		If $\nabla\Bu$ is not identically  zero,  $Y(R)>0$ for $R$ large enough. It follows from \eqref{eqA135} that one obtains
		\begin{equation}\label{eqA136}
			\left[C^{2}_{2}RY^{-2}(R)+4C_{1}
			Y^{-1}(R)\right]Y^{\prime}(R)\geq 1.
		\end{equation}
		Let $M$ be a large number satisfying $M^{-1}C^{2}_{2}\leq \frac{1}{4}$. According to Lemma \ref{Le:32}, there exists a constant $R_{0}>2$ such that $Y(R_{0})\geq M$, otherwise $\nabla\Bu\equiv0$. For every $R>R_{0}$, integrating \eqref{eqA136} over $\left[R, 2R\right]$, one gets
		\begin{equation}\label{eqA137}
			2R\cdot C^{2}_{2}\left[\dfrac{1}{Y(R)}
			-\dfrac{1}{Y(2R)}\right]+4C_{1}\ln \dfrac{Y(2R)}{Y(R)}\geq R.
		\end{equation}
		Since $Y(R)\geq M$, it holds that
		\begin{equation}\label{eqA138}
			\dfrac{Y(2R)}{Y(R)}\geq \text{exp}\left\{ \dfrac{R}{8C_{1}}\right\}.
		\end{equation}
		This implies the exponential growth of $\|\nabla\Bu\|_{L^{2}(\Omega_{R})}$. However, according to Lemma \ref{le:pebu}, $\nabla \Bu$ is uniformly bounded in
		$\Omega$. So that $\|\nabla\Bu\|_{L^{2}(\Omega_{R})}$ has at most linear growth in $R$. This contradiction implies $\nabla\Bu\equiv0$. It follows from the proof of  Lemma \ref{Le:32} that the velocity $\Bu$  must be a constant vector of the form  $(0, 0, C)$.  The proof of Theorem \ref{th:01} is completed.
	\end{proof}
	\section{flows in a slab with full slip boundary conditions}\label{Sec4}	
	In this section, we give a proof of Theorem \ref{th:02} for the case that the flows satisfy the full slip boundary conditions. One of the key ingredients is that the flow can also be regarded as a flow in a slab with periodic boundary conditions after we do a suitable extension. First, we recall a proposition, which shows that the Liouville-type theorems hold for solutions in a slab with periodic boundary conditions. The proof can be found in \cite[Section 5]{aBGWX}.
	\begin{pro} (\hspace{1sp}\cite[Theorem 1.4]{aBGWX}) \label{pro41}
		Let $\Bu$ be a bounded smooth solution to the Navier-Stokes system \eqref{eqsteadyns} in $\mathbb{R}^2\times \mathbb{T}$.
		Then $\Bu$ must be a constant vector provided that one of the following conditions holds:
		\begin{enumerate}
			
			\item[(a)] $u^\theta$ is axisymmetric, i.e., $u^\theta$ is independent of $\theta$;
			
			\item[(b)] $u^r$ is axisymmetric, i.e.,  $u^r$ is independent of $\theta$;
			
			\item[(c)] $ru^r$ converges to $0$, as $r \rightarrow +\infty$ uniformly in $\theta$, $z$;
			
			\item[(d)] $\|\Bu\|_{L^\infty(\Omega)} <2\pi. $
			
		\end{enumerate}
		Furthermore, in  cases (a), (b), and (c), the only nonzero component of the velocity field must be $u^z$, i.e., the constant vector $\Bu$ must be of the form $(0,0, c)$.
	\end{pro}
	Now we are ready to prove Theorem \ref{th:02} in the case  $\alpha=0$.
	\begin{proof}[Proof  of Theorem \ref{th:02} (full slip boundary conditions case)]
		When $\alpha=0$,	the full slip boundary conditions are
		\begin{equation}\label{totalslbou}
			u^{3}=0, \ \ \
			\partial_{x_{3}}u^{1}=0, 
			\ \ \
			\partial_{x_{3}}u^{2}=0,
			\quad \text{at} \ x_{3}=0\, \,\text{and} \,\,1.
		\end{equation}
		We make even extension for $\Bu^{h}=(u^{1},  u^{2})$,  $P$ and odd extension for $u^{3}$ in  $x_{3}$-direction. More precisely, for $\Bx\in \mathbb{R}^{2}\times [-1, 1]$, let $\Bx_{h}=(x_{1}, x_{2})$,
		\[
		\tilde{P}(\Bx_{h},x_{3})
		=\left\{
		\begin{aligned}
			&P(\Bx_{h},x_{3}), 
			\quad \ \  \text{for}\ (\Bx_{h},x_{3})\in \mathbb{R}^{2}\times [0,1],\\
			&P(\Bx_{h},-x_{3}), \quad \text{for}\ (\Bx_{h},x_{3}) \in \mathbb{R}^{2}\times [-1,0]
		\end{aligned}
		\right.
		\]
		and  	
		\[
		\tilde{\Bu}(\Bx_{h},x_{3})
		=\left\{
		\begin{aligned}
			&(\Bu^{h}(\Bx_{h}, x_{3}), u^{3}(\Bx_{h}, x_{3})),  
			\ \ \ \ \ \ \ \  \  \
			\mbox{for}\  (\Bx_{h}, x_{3})\in \mathbb{R}^{2}\times [0, 1],   \\
			&(\Bu^{h}(\Bx_{h}, -x_{3}), -u^{3}(\Bx_{h},  -x_{3})), \ \ \ \mbox{for}\ (\Bx_{h}, x_{3})\in \mathbb{R}^{2}\times [-1, 0].
		\end{aligned}
		\right.
		\]
		It can be verified that  $(\tilde{\Bu},\tilde{P})$ is a solution of the  Navier-Stokes system \eqref{eqsteadyns} in $\mathbb{R}^{2}\times (-1, 1)$ satisfying
		\[
		\tilde{\Bu}|_{x_{3}=-1}=\tilde{\Bu}|_{x_{3}=1}, \quad \tilde{P}|_{x_{3}=-1}=\tilde{P}|_{x_{3}=1}.
		\]
		We extend the solution $(\tilde{\Bu}(\Bx_{h},x_{3}), \tilde{P})$ to  a periodic solution $(\bar{\Bu}(\Bx_{h},x_{3}),\bar{P})$ in  $x_{3}$-direction with period 2.   
		Applying Proposition \ref{pro41}  finishes the proof for cases (i) and (ii) of Theorem \ref{th:02} with $\alpha=0$. Here we note that the bound for $\|\Bu\|_{L^{\infty}\left(\Omega\right)}$ is $\pi$ in case (ii), instead of $2\pi$ in Proposition \ref{pro41}, since the period for $\bar{\Bu}(\Bx_{h},x_{3})$ is $2$ while the period is $1$ in  Proposition \ref{pro41}.
	\end{proof}
	\section{Bounded solutions in a slab with  Navier boundary conditions}\label{Sec5}	
	In this section, we mainly deal with the bounded solutions of Navier-Stokes system  \eqref{eqsteadyns} in a slab satisfying the Navier boundary conditions \eqref{eqNavierbou} with $\alpha>0$.  Compared to the no-slip boundary conditions case, the Poincar\'e inequality
	\[
	\|\Bu\|_{L^{2}(\OR)}\leq C\|\partial_{z}\Bu\|_{L^{2}(\OR)}
	\]
	may not hold here. This brings some technical difficulties. 
	One of our key observations is that the boundary condition  $\Bu\cdot \Bn=0$ first gives the Poincar\'e inequality for $u^z$. Furthermore, the Navier boundary conditions, together with the divergence free property of the velocity field,  help yield some nice estimate for the whole velocity field.
	
	\subsection{$u^\theta$ or $u^r$ is axisymmetric }\label{Sec51}
	In this subsection, we  consider the flows in a slab with Navier boundary conditions  under the assumption that $u^\theta$ or $u^r$  is axisymmetric.
	\subsubsection{$u^{\theta}$ is axisymmetric}
	\label{Sec511}
	First, we  show the Liouville-type theorem  when the Dirichlet integral is finite.
	\begin{lemma}\label{Le:52}
		Let $\Bu$ be a bounded smooth  solution to the  Navier-Stokes system \eqref{eqsteadyns} in a slab $\Omega=\mathbb{R}^{2}\times (0,1)$ with Navier boundary conditions \eqref{eqNavierbou}.
		Then $\Bu$ must be  zero,
		provided that $u^{\theta}$ is independent of $\theta$
		and  $\Bu$  has a finite Dirichlet integral in the slab, i.e.,
		\begin{equation}\label{Dintass}
			\int_{\mathbb{R}^{2}\times
				(0,1)}|\nabla\Bu|^{2}\, d\Bx
			<+\infty.
		\end{equation}
	\end{lemma}
	\begin{proof}[Proof  of Lemma \ref{Le:52}]
		The proof is divided into two steps.
		
		\emph{Step 1.} \emph {Set up.}  Since  $\Bu$ is a smooth solution to the Navier-Stokes system \eqref{eqsteadyns} in $\Omega=\mathbb{R}^{2}\times (0,1)$ with Navier boundary conditions.
		Multiplying the momentum equation in \eqref{eqsteadyns} by $\varphi_{R}(r)\Bu$, one obtains
		\begin{equation}\label{eqA25}
			\int_{\Omega}-\Delta \Bu\cdot (\varphi_{R}\Bu) \,d\Bx +\int_{\Omega}( \Bu\cdot\nabla)\Bu\cdot(\varphi_{R}\Bu) \,d\Bx +\int_{\Omega}\nabla P\cdot(\varphi_{R}\Bu) \,d\Bx =0.
		\end{equation}	
		Integrating by parts and using the boundary conditions \eqref{eqNavierbou} give
		\begin{equation}\label{sec511eqANIii}
			\begin{split}
				& \int_{\Omega}-\Delta \Bu\cdot (\varphi_{R} \Bu)  \,d\Bx=-\int_{\Omega}2\text{div} \D(\Bu)\cdot (\varphi_{R}\Bu) \,d\Bx\\
				=& \int_{\Omega}2\D(\Bu): \D(\varphi_{R} \Bu) \,d\Bx-\int_{\partial\Omega}2\varphi_{R} \Bn\cdot \D(\Bu)\cdot \Bu \,dS \\
				=& \int_{\Omega}2\D (\Bu):\D (\varphi_{R}\Bu) \,d\Bx + 2\int_{\partial\Omega} \alpha\varphi_{R}\left|\Bu \right|^2 \,dS.
			\end{split}
		\end{equation}
		On the other hand, one has
		\begin{equation}\label{sec511eqidenti1}
			\int_{\Omega}-\Delta\Bu \cdot (\varphi_{R}\Bu)  \,d\Bx \\
			= \int_{\Omega}\varphi_{R}|\nabla\Bu|^{2}+\nabla \varphi_{R}\cdot \nabla   \Bu\cdot\Bu  \,d\Bx -\int_{\partial\Omega}\varphi_{R}\Bn \cdot \nabla\Bu\cdot \Bu \,dS.
		\end{equation}
		Note that on the boundary,
		\begin{equation}\nonumber
			2 \Bn \cdot \D(\Bu) \cdot \Bu - \Bn \cdot \nabla \Bu \cdot \Bu = \sum_{i,\, j=1}^3 n_{j} \partial_{x_i} u^{j} u^{i}.
		\end{equation}
		The impermeable boundary $\Bu\cdot\Bn=0$ implies that  $\nabla_{\Bt}(\Bu\cdot\Bn)=0$. Since the slab boundary is flat, one has $\nabla_{\Bt} \Bn \equiv 0$ on $ \partial \Omega$. Hence, it holds that
		\begin{equation}\label{sec511flatsid2}
			\sum_{i,\, j=1}^{3}n_{j} \partial_{x_i}u^{j} u^{i}
			= \Bu \cdot \nabla_{\Bt} (\Bu \cdot \Bn) - \Bu \cdot \left(  \Bu \cdot \nabla_{\Bt} \Bn\right)=0, \quad \text{on} \  \partial \Omega.
		\end{equation}
		From \eqref{sec511eqANIii}-\eqref{sec511flatsid2},	one has
		\begin{equation}
			\int_{\Omega}\varphi_{R}|\nabla\Bu|^{2} \,d\Bx= \int_{\Omega}2\D (\Bu):\D (\varphi_{R}\Bu) \,d\Bx-\int_{\Omega}\nabla\varphi_{R}\cdot \nabla\Bu\cdot \Bu\,d\Bx.
		\end{equation}
		Therefore,	it holds that
		\begin{equation*}\label{12eqBNIii}
			\begin{split}
				& 	\int_{\Omega}\varphi_{R}|\nabla\Bu|^{2} \,d\Bx+2\alpha\int_{\partial\Omega} \varphi_{R}\left|\Bu \right|^2 \, dS\\
				=& \int_{\Omega}2\D(\Bu): \D(\varphi_{R} \Bu) \,d\Bx+2\alpha\int_{\partial\Omega} \varphi_{R}\left|\Bu \right|^2 \, dS-\int_{\Omega}\nabla\varphi_{R}\cdot \nabla\Bu\cdot \Bu\,d\Bx\\
				=& \int_{\Omega}-\Delta \Bu\cdot (\varphi_{R}\Bu) \,d\Bx-\int_{\Omega}\nabla\varphi_{R}\cdot \nabla\Bu\cdot \Bu\,d\Bx\\
				=& \int_{\Omega}-(\Bu\cdot \nabla)\Bu\cdot (\varphi_{R}\Bu) \,d\Bx-\int_{\Omega}\nabla P\cdot (\varphi_{R}\Bu) \,d\Bx-\int_{\Omega}\nabla\varphi_{R}\cdot \nabla \Bu\cdot \Bu\,d\Bx.
			\end{split}
		\end{equation*}
		Consequently, integrating by parts yields
		\begin{equation}\label{12eqBNIii}
			\begin{split}
				& 	\int_{\Omega}\varphi_{R}|\nabla\Bu|^{2} \,d\Bx+2\alpha\int_{\partial\Omega} \varphi_{R}\left|\Bu \right|^2 \, dS\\
				=&\int_{\Omega}\frac{1}{2}|\Bu|^{2}\Bu\cdot \nabla\varphi_{R} \,d\Bx + \int_\Omega P \Bu \cdot \nabla \varphi_R \,d\Bx-\int_{\Omega}\nabla\varphi_{R}
				\cdot \nabla\Bu\cdot \Bu \,d\Bx.
			\end{split}
		\end{equation}
		Therefore,
		\begin{equation}\label{ineqinte}
			\int_{\Omega}\varphi_{R}\left| \nabla \Bu \right|^2 \,d\Bx
			\leq 
			\left|\int_{\Omega}\nabla\varphi_{R}\cdot \nabla\Bu\cdot\Bu  \,d\Bx\right|+
			\left|\int_{\Omega}\frac{1}{2}\left|\Bu \right|^2\Bu\cdot \nabla \varphi_{R}\,d\Bx\right| +\left|\int_{\Omega} P\Bu\cdot \nabla\varphi_{R} \,d\Bx\right|.
		\end{equation}	
		
		Using the divergence free property of $\Bu$ and the fact that $u^{\theta}$ is independent of $\theta$, one has
		\begin{equation}\label{eqAcaxi4}
			\partial_{r}\int_{0}^{1}ru^{r}\,dz=-\int_{0}^{1}\partial_{z}(ru^{z})\, dz=0.
		\end{equation}
		This implies
		\begin{equation}\label{eqAident4}
			\int_{0}^{1}ru^{r} \,dz=0 \quad  \text{and} \quad \int_{0}^{1} \int_{R-1}^{R}ru^{r}\, drdz=0.
		\end{equation}
		As in the proof of Lemma \ref{Le:32}, there is a vector valued function $\Ps_{R,\theta}(r,z)\in H^{1}_{0}(D_{R};\mathbb{R}^{2})$ which satisfies \eqref{eqA118} together with the estimates \eqref{eqA123} and
		\begin{equation}\label{eqB124}
			\|(\partial_{\theta}\partial_{r}, \partial_\theta \partial_z) \Ps_{R,\theta}\|_{L^{2}(\DR)} \leq
			C\|r\partial_{\theta}
			u^{r}\|_{L^{2}(\DR)}
			\leq CR^{\frac{1}{2}}\|\partial_{\theta}u^{r}\|_{L^{2}(\OR)}.
		\end{equation}
		It follows from \eqref{eqA118} and integration by parts that one has
		\begin{equation}\label{Bpreeqes}
			\begin{split}
				\int_{\Omega}P\Bu\cdot \nabla \varphi_{R} \,d\Bx
				=&
				-\int_{0}^{1}
				\int_{0}^{2\pi}
				\int_{R-1}^{R}P\cdot ru^{r} \, drd\theta dz\\
				=&
				-\int_{0}^{1}
				\int_{0}^{2\pi}
				\int_{R-1}^{R}P(\partial_{r}\Psi_{R,\theta}^{r}+\partial_{z}\Psi_{R,\theta}^{z}) \, drd\theta dz\\
				=& \int_{0}^{1}\int_{0}^{2\pi}
				\int_{R-1}^{R}(\partial_{r}P\Psi_{R,\theta}^{r}+ \partial_{z}P\Psi_{R,\theta}^{z})\, drd\theta dz.
			\end{split}
		\end{equation}
		Since $u^{\theta}$ is independent of $\theta$, we rewrite the equation \eqref{eqArP} as
		\begin{equation}\label{reeqArP}
			\begin{split}
				& \int_{0}^{1}\int_{0}^{2\pi}\int_{R-1}^{R}\partial_{r}P\Psi^{r}_{R,\theta}\, drd\theta dz\\
				=& -\int_{0}^{1}\int_{0}^{2\pi}\int_{R-1}^{R}\left[\partial_{r}u^{r}\partial_{r}\Psi^{r}_{R,\theta}+\partial_{z}u^{r}\partial_{z}\Psi^{r}_{R,\theta}
				+\dfrac{1}{r^2}\left(\partial_{\theta}u^{r}-u^{\theta}\right)\partial_{\theta}\Psi^{r}_{R,\theta}\right]\, drd\theta dz \\
				&+\int_{0}^{1}\int_{0}^{2\pi}\int_{R-1}^{R}\left[\left(\dfrac{1}{r}\partial_{r}-\dfrac{1}{r^2}\right)u^{r}
				\right]\Psi^{r}_{R,\theta}\, drd\theta dz	\\
				&-\int_{0}^{1}\int_{0}^{2\pi}\int_{R-1}^{R}\left[\left(u^{r}\partial_{r}+\dfrac{u^{\theta}}{r}\partial_{\theta}+u^{z}\partial_{z}\right)u^{r}-\dfrac{(u^{\theta})^2}{r}\right]\Psi_{R,\theta}^{r}\, drd\theta dz,
			\end{split}
		\end{equation}
		where we have used the property $\displaystyle\int_{0}^{1}\int_{0}^{2\pi}\int_{R-1}^{R}r^{-2}u^{\theta}\partial_{\theta}\Psi^{r}_{R,\theta}\, drd\theta dz=0$.
		Similarly, \eqref{eqAzP} still holds.
		
		\emph{Step 2.} \emph{Saint-Venant type estimate.}
		First, note that in the cylindrical coordinates the matrix $\nabla\Bu$ can be written as a form of tensor product as follows \cite[Appendix]{ALPY},
		\[
		\begin{split}
			\nabla\Bu
			= &
			\,\,\, \partial_{r}u^{r}\Be_{r}\otimes\Be_{r}+\left(\frac{1}{r}\partial_{\theta}u^{r}-\frac{u^{\theta}}{r}\right)\Be_{r}\otimes\Be_{\theta}+\partial_{z}u^{r}\Be_{r}\otimes\Be_{z}  \\
			&+\partial_{r}u^{\theta}\Be_{\theta}\otimes\Be_{r}+\left(\frac{1}{r}\partial_{\theta}u^{\theta}+\frac{u^{r}}{r}\right)\Be_{\theta}\otimes\Be_{\theta}+\partial_{z}u^{\theta}\Be_{\theta}\otimes\Be_{z} \\
			&+\partial_{r}u^{z}\Be_{z}\otimes\Be_{r}+\frac{1}{r}\partial_{\theta}u^{z}\Be_{z}\otimes\Be_{\theta}+\partial_{z}u^{z}\Be_{z}\otimes\Be_{z}.
		\end{split}
		\]
		So we have
		\begin{equation}\label{divfes1}
			|\partial_{r}u^{r}|+|\partial_{r}u^{\theta}|+|\partial_{r}u^{z}|
			+|r^{-1}\partial_{\theta}u^{z}|+|\partial_{z}u^{r}|+|\partial_{z}u^{\theta}|+|\partial_{z}u^{z}|\leq C|\nabla\Bu|,
		\end{equation}
		whereas
		\begin{equation}\label{divfes2}
			|r^{-1}\partial_{\theta}u^{r}-r^{-1}u^{\theta}|+|r^{-1}\partial_{\theta}u^{\theta}+r^{-1}u^{r}|\leq C|\nabla \Bu|.
		\end{equation}
		Therefore, we have
		\begin{equation}\label{divfes3}
			|r^{-1}\partial_{\theta}u^{r}|
			+|r^{-1}\partial_{\theta}u^{\theta}|\leq
			C|\nabla\Bu|+|r^{-1}u^{\theta}|+|r^{-1}u^{r}|.
		\end{equation}
		
		Now we are ready to estimate the right hand of \eqref{ineqinte}. The first two terms of \eqref{ineqinte} could be estimated  as \eqref{basces1}-\eqref{basces2}, so we just need to estimate the right hand of \eqref{Bpreeqes}.
		For the terms on the right hand of \eqref{reeqArP},	by virtue of \eqref{eqA123},  \eqref{eqB124} and \eqref{divfes1}-\eqref{divfes3}, one has
		\begin{equation}\label{estimnew1}
			\begin{split}
				&\left|\int_{0}^{1}\int_{0}^{2\pi}\int_{R-1}^{R}\dfrac{1}{r^2}\left(\partial_{\theta}u^{r}-u^{\theta}\right)\partial_{\theta}\Psi_{R,\theta}^{r}\, drd\theta dz \right|	\\
				\leq& CR^{-\frac{1}{2}}\| \nabla\Bu\|_{L^{2}(\OR)}
				\cdot R^{\frac{1}{2}} \|r^{-1}\partial_{\theta}u^{r}\|_{L^{2}(\OR)}\\
				\leq& C\| \nabla\Bu\|_{L^{2}(\OR)}\left(
				\|r^{-1}(\partial_{\theta}u^{r}-u^{\theta})\|_{L^{2}(\OR)}+\|r^{-1}u^{\theta}\|_{L^{2}(\OR)}\right)\\
				\leq& C
				\|\nabla\Bu\|_{L^{2}(\OR)}\left(
				\|\nabla\Bu\|_{L^{2}(\OR)}+R^{-\frac{1}{2}}\|\Bu\|_{L^{\infty}(\OR)}\right)\\
				\leq&  C
				\| \nabla\Bu\|_{L^{2}(\OR)},
			\end{split}
		\end{equation}
		where the last inequality is due to the finite Dirichlet integral assumption \eqref{Dintass}.
		Furthermore, one has
		\begin{equation}\label{estimnew2}
			\begin{split}
				&\left|\int_{0}^{1}\int_{0}^{2\pi}\int_{R-1}^{R}r^{-1}\left[u^{\theta}\partial_{\theta}u^{r}-(u^{\theta})^{2}\right]\Psi_{R,\theta}^{r}\, drd\theta dz \right|\\
				\leq& CR^{-\frac{1}{2}}\|\Bu\|_{L^{\infty}(\OR)}\| r^{-1}\left(\partial_{\theta}u^{r}-u^{\theta}\right)\|_{L^{2}(\OR)}
				\cdot R^{\frac{1}{2}} \|u^{r}\|_{L^{2}(\OR)}\\
				\leq & 
				CR^{\frac{1}{2}}\| \nabla\Bu\|_{L^{2}(\OR)}.
			\end{split}
		\end{equation}
		Combining the estimates  \eqref{eqA126}, \eqref{eqA128}-\eqref{eqA129} and \eqref{estimnew1}-\eqref{estimnew2}, one derives
		\begin{equation}\label{eqB130}
			\left|\int_{0}^{1}\int_{0}^{2\pi}\int_{R-1}^{R}\partial_{r}P\Psi^{r}_{R,\theta}\, dr d\theta dz\right|
			\leq CR^{\frac{1}{2}}\|\nabla\Bu\|_{L^{2}(\OR)}.
		\end{equation}
		Furthermore, note that
		\begin{equation}\label{diffurtersec521}
			\begin{split}
				&\left|\int_{0}^{1}\int_{0}^{2\pi}
				\int_{R-1}^{R}\dfrac{1}{r^2}\partial_{\theta}u^{z}\partial_{\theta}
				\Psi_{R,\theta}^{z}\,drd{\theta}dz
				\right|\\
				\leq& CR^{-\frac{1}{2}}
				\| \nabla\Bu\|_{L^{2}(\OR)}
				\cdot R^{\frac{1}{2}}\|r^{-1}\partial_{\theta}u^{r}\|_{L^{2}(\OR)}\\
				\leq& C
				\| \nabla\Bu\|_{L^{2}(\OR)}
				\left( \|r^{-1}(\partial_{\theta}u^{r}-u^{\theta})\|_{L^{2}(\OR)}+\|r^{-1}u^{\theta}\|_{L^{2}(\OR)}\right)\\
				\leq& C
				\|\nabla\Bu\|_{L^{2}(\OR)}\left(
				\|\nabla\Bu\|_{L^{2}(\OR)}+R^{-\frac{1}{2}}\|\Bu\|_{L^{\infty}(\OR)}\right)
				\\
				\leq & C
				\|\nabla\Bu\|_{L^{2}(\OR)},
			\end{split}
		\end{equation}
		where the last inequality is due to the finite Dirichlet integral assumption \eqref{Dintass}.
		The estimates for other terms  on the right hand of \eqref{eqAzP} is  quite similar to that for \eqref{eqA126} and \eqref{eqA129}.  So it holds that
		\begin{equation}\label{eqB131}
			\left|\int_{0}^{1}\int_{0}^{2\pi}\int_{R-1}^{R}\partial_{z}P\Psi^{z}_{R,\theta}\, dr d\theta dz\right|
			\leq CR^{\frac{1}{2}}\|\nabla\Bu\|_{L^{2}(\OR)}.
		\end{equation}
		Hence one arrives at
		\begin{equation}\label{sec521in}
			Y(R)\leq CR^{\frac{1}{2}} \left[Y^{\prime}(R)\right]^{\frac{1}{2}},
		\end{equation}
		where $Y(R)$ is defined in \eqref{gDinsgro}.
		The same argument as in the proof for Lemma \ref{Le:32} yields that  $\nabla\Bu\equiv 0$ and $\Bu=(0,0,C)$  when  $u^{\theta}$ is axisymmetric. This, together with the boundary  $u^{z}|_{z=0,\, 1}=0$ yields $C=0$. Hence the proof of Lemma \ref{Le:52} is completed.
	\end{proof}
	Next we show the Liouville theorem for bounded flows.
	\begin{proof}[Proof  for Case (a) of Theorem \ref{th:02} when $u^{\theta}$ is axisymmetric]
		
		The  proof is almost the same as that for Lemma \ref{Le:52}, except that
		\begin{equation}\label{Aestimnew}
			\left|\int_{0}^{1}\int_{0}^{2\pi}
			\int_{R-1}^{R}\dfrac{1}{r^2}\left(\partial_{\theta}u^{r}-u^{\theta}\right)\partial_{\theta}\Psi_{R,\theta}^{r}\, dr d\theta dz \right| \leq C
			\|\nabla\Bu\|^{2}_{L^{2}(\OR)}+C
			\|\nabla\Bu\|_{L^{2}(\OR)}
		\end{equation}
		and
		\begin{equation}\label{cadiersec521}
			\left|\int_{0}^{1}\int_{0}^{2\pi}\int_{R-1}^{R}\dfrac{1}{r^2}\partial_{\theta}u^{z}\partial_{\theta}\Psi_{R,\theta}^{z}\, drd{\theta}dz\right|
			\leq C \|\nabla\Bu\|^{2}_{L^{2}(\OR)}+C
			\|\nabla\Bu\|_{L^{2}(\OR)},
		\end{equation}
		which have been obtained in  \eqref{estimnew1} and \eqref{diffurtersec521}, respectively. According to  Lemma \ref{Le:NAvbouns} and Lemma \ref{Le:52}, following the same steps as that for  Theorem \ref{th:01}, one obtains $\Bu\equiv 0$. Therefore, the proof for  Case (a) of Theorem \ref{th:02}  when $u^{\theta}$ is axisymmetric is completed.
	\end{proof}	
	
	\subsubsection{$u^{r}$ is  axisymmetric}\label{Sec512}
	The proof is almost the same as that in Subsection \ref{Sec511}, where $u^{\theta}$ is axisymmetric.
	\begin{proof}[Proof for Case (a) of Theorem \ref{th:02} when $u^{r}$ is axisymmetric]  First we  assume that the  Dirichlet integral is finite. Using the divergence free property of $\Bu$,  for $0\leq r<\infty$, one has
		$$
		\partial_{r} \int_{0}^{1}\int_{0}^{2\pi}ru^{r}\,d\theta dz
		=-\int_{0}^{1}\int_{0}^{2\pi}\partial_{\theta}u^{\theta}+\partial_{z}(ru^{z})\, d\theta dz
		=0.
		$$
		Since $u^{r}$ is independent of $\theta$, it holds that
		\begin{equation}
			\int_{0}^{1}ru^{r}\, dz=\frac{1}{2\pi}\int_{0}^{1}\int_{0}^{2\pi}ru^{r}\, d\theta   dz=0  \quad \text{and} \quad \int_{0}^{1}\int_{R-1}^{R}u^{r}\, drdz=0.
		\end{equation}
		Similarly, as in the proof of Lemma \ref{Le:32}, there is a vector valued function $\Ps_{R}(r,z)\in H^{1}_{0}(D_{R};\mathbb{R}^{2})$ which satisfies \eqref{eqA118} together with the estimates \eqref{eqA123} and \eqref{eqB124}.
		One  obtains
		\begin{equation}\label{sec522Bpreeqes}
			\begin{split}
				\int_{\Omega}P\Bu\cdot \nabla \varphi_{R} \,d\Bx
				=&-\int_{0}^{1}
				\int_{0}^{2\pi}
				\int_{R-1}^{R}P\cdot ru^{r} \, dr d\theta dz\\
				=&-\int_{0}^{1}
				\int_{0}^{2\pi}
				\int_{R-1}^{R}P(\partial_{r}\Psi_{R}^{r}+\partial_{z}\Psi_{R}^{z}) \, dr d\theta dz\\
				=& \int_{0}^{1}\int_{0}^{2\pi}
				\int_{R-1}^{R}(\partial_{r}P\Psi_{R}^{r}+ \partial_{z}P\Psi_{R}^{z})\, drd\theta dz.
			\end{split}
		\end{equation}
		Since $u^{r}$ is independent of $\theta$,
		we rewrite the equation \eqref{eqArP} as \begin{equation}\label{sec512reeqArP1dd}
			\begin{split}
				& \int_{0}^{1}\int_{0}^{2\pi}\int_{R-1}^{R}\partial_{r}P\Psi^{r}_{R}\, drd\theta dz\\
				=& -\int_{0}^{1}\int_{0}^{2\pi}\int_{R-1}^{R}\left(\partial_{r}u^{r}\partial_{r}\Psi^{r}_{R}+\partial_{z}u^{r}\partial_{z}\Psi^{r}_{R}\right)\, drd\theta dz \\
				&+\int_{0}^{1}\int_{0}^{2\pi}\int_{R-1}^{R}\left[\left(\dfrac{1}{r}\partial_{r}-\dfrac{1}{r^2}\right)u^{r}
				-\frac{2}{r^{2}}\partial_{\theta}u^{\theta}\right]\Psi^{r}_{R}\, drd\theta dz	\\
				&-\int_{0}^{1}\int_{0}^{2\pi}\int_{R-1}^{R}\left[\left(u^{r}\partial_{r}+u^{z}\partial_{z}\right)u^{r}-\dfrac{(u^{\theta})^2}{r}\right]\Psi_{R}^{r}\, drd\theta dz.
			\end{split}
		\end{equation}
		Note that \eqref{eqAzP} still holds by replacing  $\Psi_{R, \theta}^{z}$ with $\Psi_{R}^{z}$.
		Now we are ready to estimate the terms on the right hand of \eqref{sec512reeqArP1dd}.
		Since $\Psi^{r}_{R}$ is independent of $\theta$, one has
		\begin{equation}\label{reestse52122}
			\int_{0}^{1}\int_{0}^{2\pi}\int_{R-1}^{R}\frac{2}{r^{2}}\partial_{\theta}u^{\theta}\Psi^{r}_{R}\, drd\theta dz=0.
		\end{equation}
		This, together with \eqref{eqA126}, \eqref{eqA128} and \eqref{eqA129} gives
		\begin{equation}\label{sec52eqB130}
			\left|\int_{0}^{1}\int_{0}^{2\pi}\int_{R-1}^{R}\partial_{r}P\Psi^{r}_{R}
			\, drd\theta dz\right|
			\leq CR^{\frac{1}{2}}\|\nabla\Bu\|_{L^{2}(\OR)}.
		\end{equation}
		Similarly, as in the proof  of Lemma \ref{Le:52},   one  obtains
		\begin{equation}\label{sec522eqB131}
			\left|\int_{0}^{1}\int_{0}^{2\pi}\int_{R-1}^{R}\partial_{z}P\Psi^{z}_{R}\, dr d\theta dz\right|
			\leq CR^{\frac{1}{2}}\|\nabla\Bu\|_{L^{2}(\OR)},
		\end{equation}
		provided that the  Dirichlet integral is finite.
		
		Combining the estimates  \eqref{basces1}-\eqref{basces2} and \eqref{sec52eqB130}-\eqref{sec522eqB131}, one arrives at \eqref{sec521in} with $Y(R)$ defined in \eqref{gDinsgro}. According to the proof for Lemma \ref{Le:32}, one obtains  $\nabla\Bu\equiv 0$. Furthermore, it follows from the axisymmetry  of $u^{r}$ and the Navier boundary conditions \eqref{eqA10} that the solution $\Bu$ must be zero. So we finish the proof for the  Liouville-type theorem when $u^{r}$ is axisymmetric and the   Dirichlet integral is finite. Following the same steps as in the proof for Case (a) of Theorem \ref{th:02} with   axisymmetric $u^{\theta}$, we can remove the  finite Dirichlet integral assumption.
		This finishes the proof for Case (a) of Theorem \ref{th:02} when $u^{r}$ is axisymmetric.
	\end{proof}
	\subsection{General 3D solutions  with $ru^{r}$ decaying to zero}\label{Sec52}
	
	This subsection is devoted to the study for general solutions. Since $\partial_{\theta} P$ doesn't have the same scaling as $\partial_{r} P$ and $\partial_{z} P$, so we have to deal with the term $\partial_{\theta} P$ carefully.	
	\begin{proof}[Proof for Case (b) of Theorem \ref{th:02}]
		The proof is divided into two steps.
		
		\emph{Step 1.} \emph {Set up.}
		The inequality \eqref{ineqinte} still holds. In this case,
		\begin{equation}\label{eqA55}
			\int_{\Omega}P\Bu \cdot \nabla \varphi_{R}\, d\Bx
			=- \int_{0}^{1}\int_{0}^{2\pi}
			\int_{R-1}^{R}Pu^{r}r \, dr d \theta dz.
		\end{equation}
		It follows from the divergence free condition that for every fixed $r\geq 0$, one has	
		\begin{equation}\label{eqA56}
			\partial_{r}\int_{0}^{1}\int_{0}^{2\pi}r u^{r}\, d \theta dz
			=-\int_{0}^{1}\int_{0}^{2\pi}\partial_{\theta} u^{\theta} \,d \theta dz
			-\int_{0}^{1} \int_{0}^{2\pi}
			\partial_{z}(ru^{z})\, d \theta dz
			=0.
		\end{equation}
		And then it holds that
		\begin{equation}\label{eqA57}
			\int_{0}^{1}\int_{0}^{2\pi}ru^{r}\, d\theta dz=0,\,\,\,
			\int_{0}^{1}\int_{0}^{2\pi}u^{r}\, d\theta dz=0
			\,\,\,\,\, \text{and}  \,\,\,\,\,
			\int_{0}^{1}\int_{0}^{2\pi}\int_{R-1}^{R}ru^{r}\, drd\theta dz=0.
		\end{equation}
		By virtue of  Lemma \ref{Bogovskii},
		there exists a vector valued function $\Ps_{R}(r,\theta,z)\in H^{1}_{0}(\DR;\mathbb{R}^{3})$ satisfying
		\begin{equation}\label{eqA59}
			\partial_{r}\Psi_{R}^{r}
			+\partial_{\theta}\Psi_{R}^{\theta}
			+\partial_{z}\Psi_{R}^{z}
			=ru^{r}
		\end{equation}
		and
		\begin{equation}\label{eqA60}
			\|(\partial_{r}, \partial_{\theta}, \partial_{z})\Ps_{R}\|_{L^{2}(\DR)}
			\leq C\|ru^{r}\|_{L^{2}(\DR)}
			\leq CR^{\frac{1}{2}}\|u^{r}\|_{L^{2}(\OR)}.
		\end{equation}
		From \eqref{eqA55} and \eqref{eqA59},   one obtains
		\begin{equation}\label{eqA61}
			\begin{split}
				\int_{\Omega}P\Bu\cdot \nabla \varphi_{R}\, d\Bx
				=&-\int_{0}^{1}
				\int_{0}^{2\pi}
				\int_{R-1}^{R}P(\partial_{r}\Psi_{R}^{r}+\partial_{\theta}\Psi_{R}^{\theta}+\partial_{z}\Psi_{R}^{z})\,dr d\theta dz\\
				=& \int_{0}^{1}\int_{0}^{2\pi}
				\int_{R-1}^{R}(\partial_{r}P\Psi_{R}^{r}+\partial_{\theta}P\Psi_{R}^{\theta}+ \partial_{z}P\Psi_{R}^{z})\, dr d\theta dz.
			\end{split}
		\end{equation}
		Furthermore, it follows from the momentum equations in \eqref{eqA8} that one has
		\begin{equation}\label{eqA62}
			\begin{split}
				& \int_{0}^{1}\int_{0}^{2\pi}\int_{R-1}^{R}\partial_{r}P\Psi^{r}_{R} \, dr d \theta dz\\
				=& -\int_{0}^{1}\int_{0}^{2\pi}\int_{R-1}^{R}\left[\partial_{r}u^{r}\partial_{r}\Psi^{r}_{R}+\partial_{z}u^{r}\partial_{z}\Psi^{r}_{R}+\dfrac{1}{r^2}\left(\partial_{\theta}u^{r}-2u^{\theta}\right)\partial_{\theta}\Psi^{r}_{R}\right]\, dr d\theta dz \\
				&+\int_{0}^{1}\int_{0}^{2\pi}\int_{R-1}^{R}\left[\left(\dfrac{1}{r}\partial_{r}-\dfrac{1}{r^2}\right)u^{r}\right]\Psi^{r}_{R} \, dr d\theta dz	\\
				&-\int_{0}^{1}\int_{0}^{2\pi}\int_{R-1}^{R}\left[\left(u^{r}\partial_{r}+\dfrac{u^{\theta}}{r}\partial_{\theta}+u^{z}\partial_{z}\right)u^{r}-\dfrac{(u^{\theta})^2}{r}\right]\Psi_{R}^{r}\, dr d\theta dz,
			\end{split}
		\end{equation}
		\begin{equation}\label{eqA63}
			\begin{split}
				& \int_{0}^{1}\int_{0}^{2\pi}\int_{R-1}^{R}\partial_{\theta}P\Psi^{\theta}_{R} \,dr d\theta dz\\
				=& -\int_{0}^{1}\int_{0}^{2\pi}\int_{R-1}^{R}\left[r\left(\partial_{r}u^{\theta}\partial_{r}\Psi^{\theta}_{R}+\partial_{z}u^{\theta}\partial_{z}\Psi^{\theta}_{R}\right)+r^{-1}(\partial_{\theta}u^{\theta}+2u^{r})\partial_{\theta}\Psi^{\theta}_{R}\right] \, dr d\theta dz \\
				&-\int_{0}^{1}\int_{0}^{2\pi}\int_{R-1}^{R}\dfrac{u^{\theta}}{r}\Psi^{\theta}_{R}\, dr d\theta dz	\\
				&-\int_{0}^{1}\int_{0}^{2\pi}\int_{R-1}^{R}\left[r\left(u^{r}\partial_{r}+\dfrac{u^{\theta}}{r}\partial_{\theta}+u^{z}\partial_{z}\right)u^{\theta}+u^{\theta}u^{r}\right]\Psi_{R}^{\theta} \, dr d\theta dz
			\end{split}
		\end{equation}
		and
		\begin{equation}\label{eqA64}
			\begin{split}
				& \int_{0}^{1}\int_{0}^{2\pi}\int_{R-1}^{R}\partial_{z}P\Psi^{z}_{R} \, drd\theta dz\\
				=&	-\int_{0}^{1}\int_{0}^{2\pi}\int_{R-1}^{R}\left(\partial_{r}u^{z}\partial_{r}\Psi_{R}^{z}+\partial_{z}u^{z}\partial_{z}\Psi_{R}^{z}+\dfrac{1}{r^2}\partial_{\theta}u^{z}\partial_{\theta}\Psi_{R}^{z}\right)\, drd{\theta}dz\\
				&-\int_{0}^{1}\int_{0}^{2\pi}\int_{R-1}^{R}\left[\left(u^{r}\partial_{r}+\dfrac{u^{\theta}}{r}\partial_{\theta}+u^{z}\partial_{z}-\dfrac{1}{r}\partial_{r}\right)u^{z}\right]\Psi_{R}^{z}\, drd\theta dz.
			\end{split}
		\end{equation}
		\emph{Step 2.} \emph{Saint-Venant type estimate.} Now we estimate the right hand of \eqref{ineqinte}, \eqref{eqA62}, \eqref{eqA63}, \eqref{eqA64} carefully.  First, due to the boundedness of the velocity, one has
		\begin{equation}\label{eqA65}
			\left|\int_{\Omega}\nabla\varphi_{R}\cdot \nabla\Bu\cdot\Bu \,d\Bx \right|
			\leq C\|\nabla\Bu\|_{L^{2}(\OR)}\|\Bu\|_{L^{2}(\OR)}
			\leq CR^{\frac{1}{2}}\|\nabla\Bu\|_{L^{2}(\OR)}
		\end{equation}
		and
		\begin{equation}\label{eqA66}
			\left|\int_{\Omega}\frac{1}{2}\left|\Bu \right|^2\Bu\cdot \nabla \varphi_{R}\, d\Bx\right|
			\leq C\|\Bu\|^{2}_{L^{\infty}(\OR)}\|u^{r}\|_{L^{1}(\OR)}
			\leq
			CR\|u^{r}\|_{L^{\infty}(\OR)},
		\end{equation}
		where $\OR$ is defined at the beginning of Section \ref{Sec2}.
		Regarding  \eqref{eqA62}, one has
		\begin{equation}\label{eqA67}
			\begin{split}
				&\left| \int_{0}^{1}\int_{0}^{2\pi}\int_{R-1}^{R}\left(\partial_{r}u^{r}\partial_{r}\Psi_{R}^{r}+\partial_{z}u^{r}\partial_{z}\Psi_{R}^{r}\right)\, dr d \theta dz\right|\\
				\leq& CR^{-\frac{1}{2}}\|\nabla\Bu\|_{L^{2}(\OR)}\cdot R^{\frac{1}{2}}\|u^{r}\|_{L^{2}(\OR)}\\
				\leq&  CR^{\frac{1}{2}}\|\nabla\Bu\|_{L^{2}(\OR)} \|u^{r}\|_{L^{\infty}(\OR)},
			\end{split}
		\end{equation}
		
		\begin{equation}\label{eqAB6712}
			\begin{split}
				&\left| \int_{0}^{1}\int_{0}^{2\pi}\int_{R-1}^{R}\left[r^{-2}\left(\partial_{\theta}u^{r}-2u^{\theta}\right)\right]\partial_{\theta}\Psi_{R}^{r}\, dr d\theta dz\right|\\
				\leq& C\left(R^{-\frac{3}{2}}\|\nabla\Bu\|_{L^{2}(\OR)}+R^{-\frac{5}{2}}\|u^{\theta}\|_{L^{2}(\OR)}\right)\cdot R^{\frac{1}{2}}\|u^{r}\|_{L^{2}(\OR)}\\
				\leq&
				CR^{-\frac{1}{2}}\|\nabla\Bu\|_{L^{2}(\OR)}\|u^{r}\|_{L^{\infty}(\OR)}+CR^{-1}\|u^{r}\|_{L^{\infty}(\OR)}
			\end{split}
		\end{equation}
		and
		\begin{equation}\label{eqA68}
			\begin{split}
				&\left|\int_{0}^{1}\int_{0}^{2\pi}\int_{R-1}^{R}\left[\left(\dfrac{1}{r}\partial_{r}-\dfrac{1}{r^2}\right)u^{r}\right]\Psi_{R}^{r}\,dr d\theta dz \right|\\
				\leq& CR^{-\frac{3}{2}}\| \nabla\Bu\|_{L^{2}(\OR)}
				\cdot R^{\frac{1}{2}} \|u^{r}\|_{L^{2}(\OR)} 	\\
				\leq& CR^{-\frac{1}{2}}\| \nabla\Bu\|_{L^{2}(\OR)}\|u^{r}\|_{L^{\infty}(\OR)}.
			\end{split}
		\end{equation}
		Furthermore, it holds that
		\begin{equation}\label{eqA69}
			\begin{split}
				&\left|\int_{0}^{1}\int_{0}^{2\pi}\int_{R-1}^{R}\left[\left(u^{r}\partial_{r}+u^{z}\partial_{z}\right)u^{r}\right]\Psi_{R}^{r}\, drd\theta dz \right|	\\
				\leq& C\| (u^{r},u^{z})\|_{L^{\infty}(\DR)}
				\| (\partial_{r},  \partial_{z})u^{r}\|_{L^{2}(\DR)}\|\Psi_{R}^{r}\|_{L^{2}(\DR)}\\
				\leq & CR^{-\frac{1}{2}}\| \nabla\Bu\|_{L^{2}(\OR)}\cdot R^{\frac{1}{2}}\|u^{r}\|_{L^{2}(\OR)}\\
				\leq & CR^{\frac{1}{2}}\| \nabla\Bu\|_{L^{2}(\OR)}\|u^{r}\|_{L^{\infty}(\OR)}
			\end{split}
		\end{equation}
		and
		\begin{equation}\label{eqA70}
			\begin{split}
				&\left|\int_{0}^{1}\int_{0}^{2\pi}\int_{R-1}^{R}\left[r^{-1}\left(u^{\theta}\partial_{\theta}u^{r}-u^{\theta}u^{\theta}\right)\right]\Psi_{R}^{r}\, drd\theta dz \right|\\	
				\leq& C\| u^{\theta}\|_{L^{\infty}(\OR)}\cdot R^{-\frac{1}{2}}\|\nabla\Bu\|_{L^{2}(\OR)}\cdot R^{\frac{1}{2}}\|u^{r}\|_{L^{2}(\OR)}\\
				\leq& CR^{\frac{1}{2}}\| \nabla\Bu\|_{L^{2}(\OR)}\|u^{r}\|_{L^{\infty}(\OR)}.
			\end{split}
		\end{equation}
		Collecting the estimates \eqref{eqA67}-\eqref{eqA70}, one obtains
		\begin{equation}\label{eqA71}
			\left|\int_{0}^{1}\int_{0}^{2\pi}
			\int_{R-1}^{R}\partial_{r}P\Psi^{r}_{R}\,dr d\theta dz\right|\leq CR^{\frac{1}{2}}\| \nabla\Bu \|_{L^{2}(\OR)}\|u^{r}\|_{L^{\infty}(\OR)}+CR^{-1}\|u^{r}\|_{L^{\infty}(\OR)}.
		\end{equation}
		Similarly,	 it can be proved that
		\begin{equation}\label{eqA77}
			\left|\int_{0}^{1}\int_{0}^{2\pi}
			\int_{R-1}^{R}\partial_{z}P\Psi^{\theta}_{R}\, drd\theta dz\right|
			\leq CR^{\frac{1}{2}}\|\nabla\Bu\|_{L^{2}(\OR)}\|u^{r}\|_{L^{\infty}(\OR)}.
		\end{equation}
		Next we estimate the right hand of \eqref{eqA63}. It follows from \eqref{divfes1}-\eqref{divfes3}  and \eqref{eqA60} that one has
		\begin{equation}\label{eqA72}
			\begin{split}
				&\left| \int_{0}^{1}\int_{0}^{2\pi}\int_{R-1}^{R}\left[r\left(\partial_{r}u^{\theta}\partial_{r}\Psi_{R}^{\theta}+\partial_{z}u^{\theta}\partial_{z}\Psi_{R}^{\theta}\right)\right]\,drd\theta dz\right|\\
				\leq& CR\|(\partial_{r},\partial_{z})u^{\theta}\|_{L^{2}(\DR)}\|(\partial_{r},\partial_{z})\Psi^{\theta}_{R}\|_{L^{2}(\DR)}\\
				\leq&  CR^{\frac{1}{2}} \|\nabla
				\Bu\|_{L^{2}(\OR)}\cdot R^{\frac{1}{2}} \| u^{r}\|_{L^{2}(\OR)}\\
				\leq &
				CR^{\frac{3}{2}} \|\nabla
				\Bu\|_{L^{2}(\OR)} \| u^{r}\|_{L^{\infty}(\OR)},
			\end{split}
		\end{equation}
		
		\begin{equation}\label{eqA73}
			\begin{split}
				&\left|\int_{0}^{1}\int_{0}^{2\pi}\int_{R-1}^{R}\left[r^{-1}\left(\partial_{\theta}u^{\theta}+2u^{r}\right)\partial_{\theta}\Psi_{R}^{\theta}\right]\, drd\theta dz \right| \\
				\leq&
				C\left(\|r^{-1}(\partial_{\theta}u^{\theta}+u^{r})\|_{L^{2}(\DR)}
				+R^{-1}\|u^{r}\|_{L^{2}(\DR)}\right)\|\partial_{\theta}\Psi_{R}^{\theta}\|_{L^{2}(\DR)}\\
				\leq&
				C \left(R^{-\frac{1}{2}} \|\nabla\Bu\|_{L^{2}(\OR)}+R^{-1}\|u^{r}\|_{L^{\infty}(\OR)}\right)\cdot R^{\frac{1}{2}}\|u^{r}\|_{L^{2}(\OR)}\\
				\leq&
				C R^{\frac{1}{2}} \|\nabla\Bu\|_{L^{2}(\OR)}\|u^{r}\|_{L^{\infty}(\OR)}+C\|u^{r}\|_{L^{\infty}(\OR)}
			\end{split}
		\end{equation}
		and
		\begin{equation}\label{reeqB73}
			\begin{split}
				&\left|\int_{0}^{1}\int_{0}^{2\pi}\int_{R-1}^{R}r^{-1}u^{\theta}\Psi_{R}^{\theta}\, drd\theta dz \right|	\\
				\leq&
				CR^{-1}\|u^{\theta}\|_{L^{2}(\DR)}
				\|\Psi_{R}^{\theta}\|_{L^{2}(\DR)}\\
				\leq &
				C R^{-1} \|u^{\theta}\|_{L^{\infty}(\OR)}
				\cdot   R^{\frac{1}{2}}\|u^{r}\|_{L^{2}(\OR)}\\
				\leq & C\|u^{r}\|_{L^{\infty}(\OR)}.
			\end{split}
		\end{equation}
		Furthermore, we have
		\begin{equation}\label{eqA74}
			\begin{split}
				&\left| \int_{0}^{1}\int_{0}^{2\pi}\int_{R-1}^{R}(u^{r}u^{\theta}+u^{\theta}\partial_{\theta}u^{\theta})\Psi_{R}^{\theta}\, drd\theta dz\right|\\
				\leq& CR^{\frac{1}{2}}\|r^{-1}(u^{r}+\partial_{\theta}u^{\theta})\|_{L^{2}(\OR)}\|u^{\theta}\|_{L^{\infty}(\DR)}\|\Psi^{\theta}_{R}\|_{L^{2}(\DR)}\\
				\leq& CR^{\frac{1}{2}}\|\nabla\Bu\|_{L^{2}(\OR)}\cdot R^{\frac{1}{2}}\|u^{r}\|_{L^{2}(\OR)}\\
				\leq&  CR^{\frac{3}{2}} \|\nabla
				\Bu\|_{L^{2}(\OR)}  \| u^{r}\|_{L^{\infty}(\OR)}
			\end{split}
		\end{equation}
		and
		\begin{equation}\label{eqA75}
			\begin{split}
				&\left|\int_{0}^{1}\int_{0}^{2\pi}\int_{R-1}^{R}\left[r\left(u^{r}\partial_{r}+u^{z}\partial_{z}\right)u^{\theta}\right]\Psi_{R}^{\theta}\, drd\theta dz \right|	
				\\
				\leq& CR\| (u^{r}, u^{z})\|_{L^{\infty}(\DR)}
				\| (\partial_{r},  \partial_{z})u^{\theta}\|_{L^{2}(\DR)}\|\Psi_{R}^{\theta}\|_{L^{2}(\DR)}	\\
				\leq & CR^{\frac{1}{2}}
				\|\nabla\Bu\|_{L^{2}(\OR)}\cdot R^{\frac{1}{2}}\|u^{r}\|_{L^{2}(\OR)}\\
				\leq & CR^{\frac{3}{2}}\| \nabla\Bu\|_{L^{2}(\OR)}\|u^{r}\|_{L^{\infty}(\OR)}.
			\end{split}
		\end{equation}
		Collecting the estimates \eqref{eqA72}-\eqref{eqA75}, one obtains
		\begin{equation}\label{eqA76}
			\left|\int_{0}^{1}\int_{0}^{2\pi}\int_{R-1}^{R}\partial_{\theta}P\Psi^{\theta}_{R}\, drd\theta dz\right|
			\leq CR^{\frac{3}{2}}\| \nabla\Bu\|_{L^{2}(\OR)}\|u^{r}\|_{L^{\infty}(\OR)}+C\|u^{r}\|_{L^{\infty}(\OR)}.
		\end{equation}
		Combining the estimates \eqref{eqA65}-\eqref{eqA66}, \eqref{eqA71}-\eqref{eqA77} and \eqref{eqA76}, it can be shown that
		\begin{equation}\label{eqA78}
			Y(R)\leq CR^{\frac{1}{2}}\|\nabla\Bu\|_{L^{2}(\OR)}+CR\|u^{r}\|_{L^{\infty}(\OR)}+CR^{\frac{3}{2}}\|\nabla\Bu\|_{L^{2}(\OR)}\|u^{r}\|_{L^{\infty}(\OR)},
		\end{equation}
		where $Y(R)$ is defined in \eqref{gDinsgro}.
		Since $ru^{r}$ is bounded,  one obtains that
		\begin{equation}\label{eqA80}
			Y(R)\leq C_{1}R^{\frac{1}{2}}[Y^{\prime}(R)]^{\frac{1}{2}}+C_{2}R\|u^{r}\|_{L^{\infty}(\OR)}.
		\end{equation}
		Assume that $\nabla\Bu$ is not identically equal to zero, there exists a constant $R_{0}$ large enough, such that $Y(R_{0})> 0$. Since $ru^{r}$ converges to zero uniformly in $\theta$, $z$, there exists some $R_{1}>R_{0}$ such that $Y(R_{0})\geq 2 C_{2}R\|u^{r}\|_{L^{\infty}(\OR)}$ for every $R\geq R_{1}$.
		This implies 
		\[
		Y(R)\leq 2C_{1}R^{\frac{1}{2}}[Y^{\prime}(R)]^{\frac{1}{2}}, \quad  R\geq R_{1},
		\]
		which leads to a contradiction as in the proof  for Lemma \ref{Le:32}. Hence $\Bu$ equals to zero, thanks to the Navier boundary conditions \eqref{eqA10} and  $ru^{r}$ is bounded. This completes the proof for Case (b) of Theorem \ref{th:02}.
	\end{proof}
	\subsection{General 3D flows  when $\|\Bu\|_{L^{\infty}(\Omega)}$ is not big}\label{Sec54}	
	In this subsection,  if the velocity $\Bu$ is not too big, we prove that the  flow in a slab with Navier boundary conditions  must be a Poiseuille type flow.
	\begin{proof}[Proof for case (iii) of Theorem \ref{th:02}]
		Assume that $\Bu$ is a smooth solution to the Navier-Stokes system \eqref{eqsteadyns} in a slab $\Omega=\mathbb{R}^{2}\times (0,1)$ with Navier boundary conditions \eqref{eqNavierbou} where $\alpha>0$.
		Taking the $x_{i}$-derivative ($i=1, 2$) of the momentum equation in \eqref{eqsteadyns}, one has
		\begin{equation}\label{sec6derivxeq}
			-\Delta \partial_{x_i}\Bu + (\partial_{x_i}\Bu \cdot \nabla )\Bu +(\Bu \cdot \nabla )\partial_{x_i}\Bu+ \nabla \partial_{x_i}P  = 0.
		\end{equation}
		Multiplying the equation  \eqref{sec6derivxeq} by $\partial_{x_i}\Bu\varphi_{R}(r)$ and integrating over $\Omega$,  one obtains
		\begin{equation}\label{sec6eqmul}
			\begin{split}
				&\int_{\Omega}-\Delta \partial_{x_i}\Bu\cdot \partial_{x_i}\Bu\varphi_{R} \,d\Bx +\int_{\Omega}( \partial_{x_i}\Bu\cdot\nabla)\Bu\cdot
				\partial_{x_i}\Bu \varphi_{R} \,d\Bx\\
				=&\frac{1}{2}\int_{\Omega}\left(\Bu\cdot \nabla \varphi_{R}\right)|\partial_{x_i} \Bu|^{2}\,d\Bx+\int_{\Omega}	\partial_{x_i} P \partial_{x_i}\Bu \cdot \nabla\varphi_{R} \,d\Bx.
			\end{split}
		\end{equation}	
		If $\Bu$ in \eqref{sec511eqANIii}-\eqref{sec511flatsid2} is   replaced  by $\partial_{x_i}\Bu$, then one has
		\begin{equation}\label{eqse6Dibem}
			\begin{split}
				&\int_{\Omega}-\Delta \partial_{x_i}\Bu\cdot \partial_{x_i}\Bu\varphi_{R} \,d\Bx\\
				=&
				\int_{\Omega}\varphi_{R}|\nabla\partial_{x_i}\Bu|^{2} \,d\Bx+2\alpha\int_{\partial\Omega} \varphi_{R}\left|\partial_{x_i}\Bu \right|^2 \, dS+\int_{\Omega}\nabla\varphi_{R}\cdot \nabla \partial_{x_i}\Bu\cdot \partial_{x_i}\Bu\,d\Bx.
			\end{split}
		\end{equation}
		Consequently, combining \eqref{sec6eqmul} and \eqref{eqse6Dibem} give
		\begin{equation}\label{2sec6eqmul223}
			\begin{split}
				&\int_{\Omega}\varphi_{R}|\nabla\partial_{x_i}\Bu|^{2} \,d\Bx+2\alpha\int_{\partial\Omega} \varphi_{R}\left|\partial_{x_i}\Bu \right|^2 \, dS+\int_{\Omega}( \partial_{x_i}\Bu\cdot\nabla)\Bu\cdot
				\partial_{x_i}\Bu \varphi_{R} \,d\Bx\\
				=&\frac{1}{2}\int_{\Omega}\left(\Bu\cdot \nabla \varphi_{R}\right)|\partial_{x_i} \Bu|^{2}\,d\Bx+\int_{\Omega}	\partial_{x_i} P \partial_{x_i}\Bu \cdot \nabla\varphi_{R} \,d\Bx-\int_{\Omega}\nabla\varphi_{R}\cdot \nabla \partial_{x_i}\Bu\cdot \partial_{x_i}\Bu\,d\Bx.
			\end{split}
		\end{equation}	
		For the third term on the left hand of \eqref{2sec6eqmul223},  integrating by parts yields
		\begin{equation}
			\int_{\Omega}(\partial_{x_i} \Bu\cdot \nabla)\Bu\cdot \partial_{x_i} \Bu\varphi_{R}\,d\Bx=-\int_{\Omega}(\partial_{x_i} \Bu\cdot \nabla\partial_{x_i})\Bu\cdot  \Bu\varphi_{R}\,d\Bx-\int_{\Omega}(\partial_{x_i} \Bu\cdot \nabla\varphi_{R})(\Bu\cdot \partial_{x_i} \Bu)\,d\Bx.
		\end{equation}
		In fact, one obtains
		\begin{equation}\label{eqedpruvarf}
			\begin{split}
				&\left|\int_{\Omega}(\partial_{x_i} \Bu\cdot \nabla\partial_{x_i})\Bu\cdot  \Bu\varphi_{R}\,d\Bx\right|\\
				\leq& \|\nabla \partial_{x_i}\Bu \sqrt{\varphi_{R}}\|_{L^{2}(\Omega)}\|\partial_{x_i}\Bu \sqrt{\varphi_{R}}\|_{L^{2}(\Omega)}
				\|\Bu\|_{L^{\infty}(\Omega)}\\
				\leq& \|\nabla \partial_{x_i}\Bu \sqrt{\varphi_{R}}\|_{L^{2}(\Omega)}\cdot 2
				\left(\|\partial_{x_3}\partial_{x_i}\Bu \sqrt{\varphi_{R}}\|_{L^{2}(\Omega)}+\|
				\partial_{x_i}\Bu \sqrt{\varphi_{R}}\|_{L^{2}(\partial\Omega)}\right)
				\|\Bu\|_{L^{\infty}(\Omega)}\\
				\leq& 	2\|\Bu\|_{L^{\infty}(\Omega)} \left(\|\nabla\partial_{x_i}\Bu \sqrt{\varphi_{R}}\|_{L^{2}(\Omega)}+\|
				\partial_{x_i}\Bu \sqrt{\varphi_{R}}\|_{L^{2}(\partial\Omega)}\right)^{2}\\
				\leq& 	4\|\Bu\|_{L^{\infty}(\Omega)} \left(\|\nabla\partial_{x_i}\Bu \sqrt{\varphi_{R}}\|^{2}_{L^{2}(\Omega)}+\|
				\partial_{x_i}\Bu \sqrt{\varphi_{R}}\|^{2}_{L^{2}(\partial\Omega)}\right),
			\end{split}
		\end{equation}
		where the following Poincar\'e type inequality
		\[
		\|\partial_{x_i}\Bu \sqrt{\varphi_{R}}\|_{L^{2}(\Omega)}\leq 2
		\left(\|\partial_{x_3}\partial_{x_i}\Bu \sqrt{\varphi_{R}}\|_{L^{2}(\Omega)}+
		\| \partial_{x_i}\Bu \sqrt{\varphi_{R}}\|_{L^{2}(\partial\Omega)}\right)
		\]
		has been used to get the second inequality in \eqref{eqedpruvarf}.
		Furthermore,  one  has  the Poincar\'e type inequality
		\begin{equation}\label{se6POINTY2}
			\|\partial_{x_{i}}\Bu\|_{L^{2}(\OR)}\leq
			C\left(\|\partial_{x_3}\partial_{x_{i}}\Bu\|_{L^{2}(\OR)}+\|\partial_{x_{i}}\Bu\|_{L^{2}(\partial\OR\cap \partial \Omega)}\right),
		\end{equation}
		where $C$ is a universal constant.
		By the Poincar\'e type inequality \eqref{se6POINTY2} and Lemma \ref{Le:NAvbouns}, it holds that
		\begin{equation}
			\begin{split}
				\left|\int_{\Omega}(\partial_{x_i} \Bu\cdot \nabla\varphi_{R})(\Bu\cdot \partial_{x_i} \Bu)\,d\Bx\right|
				\leq& \|\partial_{x_i}\Bu\|^{2}_{L^{2}(\OR)}\|\Bu\|_{L^{\infty}(\OR)}\\
				\leq& CR^{\frac{1}{2}}\left(\|\nabla\partial_{x_{i}}\Bu\|_{L^{2}(\OR)}+\|\partial_{x_{i}}\Bu\|_{L^{2}(\partial\OR\cap \partial \Omega)}\right),
			\end{split}
		\end{equation}
		\begin{equation}
			\begin{split}
				\left|\int_{\Omega}(\Bu\cdot \nabla\varphi_{R})|\partial_{x_i} \Bu|^{2}\,d\Bx\right|
				\leq& \|\Bu\|_{L^{\infty}(\OR)} \|\partial_{x_i}\Bu\|^{2}_{L^{2}(\OR)}\\
				\leq& CR^{\frac{1}{2}}\left(\|\nabla\partial_{x_{i}}\Bu\|_{L^{2}(\OR)}+\|\partial_{x_{i}}\Bu\|_{L^{2}(\partial\OR\cap \partial \Omega)}\right)
			\end{split}
		\end{equation}
		and
		\begin{equation}
			\left|\int_{\Omega}\nabla\varphi_{R}\cdot \nabla \partial_{x_i}\Bu\cdot \partial_{x_i}\Bu\,d\Bx\right|\leq \|\nabla \partial_{x_i}\Bu\|_{L^{2}(\OR)} \|\partial_{x_i}\Bu\|_{L^{2}(\OR)}\leq CR^{\frac{1}{2}}\|\nabla\partial_{x_i}\Bu\|_{L^{2}(\OR)}.
		\end{equation}
		Using the momentum equation in \eqref{eqsteadyns} gives
		\begin{equation}
			\begin{split}
				&\int_{\Omega}	\partial_{x_i} P \partial_{x_i}\Bu \cdot \nabla\varphi_{R} \,d\Bx=\int_{\Omega}\Delta u^{i} \partial_{x_i}\Bu \cdot \nabla\varphi_{R} \,d\Bx-\int_{\Omega}	\left(\Bu\cdot \nabla u^{i}\right) \partial_{x_i}\Bu \cdot \nabla\varphi_{R} \,d\Bx\\
				=&\int_{\Omega}	(\partial^{2}_{x_1}+\partial^{2}_{x_2}) u^{i} \partial_{x_i}\Bu \cdot \nabla\varphi_{R} \,d\Bx-\int_{\Omega}	\partial_{x_3} u^{i} \partial_{x_3} \partial_{x_i}\Bu \cdot \nabla\varphi_{R} \,d\Bx-\int_{\Omega}	(\Bu\cdot \nabla u^{i}) \partial_{x_i}\Bu \cdot \nabla\varphi_{R} \,d\Bx.
			\end{split}
		\end{equation}
		Consequently, one derives
		\begin{equation}
			\begin{split}
				&\left|\int_{\Omega}	\partial_{x_i} P \partial_{x_i}\Bu \cdot \nabla\varphi_{R} \,d\Bx\right|\\
				\leq& CR^{\frac{1}{2}}\left[(\|\nabla\partial_{x_1}\Bu\|_{L^{2}(\OR)}+\|\partial_{x_1}\Bu\|_{L^{2}(\partial\OR\cap \partial \Omega)})+(\|\nabla\partial_{x_2}\Bu\|_{L^{2}(\OR)}+\|\partial_{x_2}\Bu\|_{L^{2}(
					\partial\OR\cap \partial \Omega)})\right].
			\end{split}
		\end{equation}
		Due to $\|\Bu\|_{L^{\infty}(\Omega)}<	\min\left\{\dfrac{\alpha}{2}, \, \dfrac{1}{4}\right\}$,  it can be shown that
		\begin{equation}\label{sec6inequ156}
			\begin{split}
				&\int_{\Omega}\varphi_{R}|\nabla\partial_{x_i}\Bu|^{2} \,d\Bx+2\alpha\int_{\partial\Omega} \varphi_{R}\left|\partial_{x_i}\Bu \right|^2 \, dS\\
				\leq& CR^{\frac{1}{2}}\left[(\|\nabla\partial_{x_1}\Bu\|_{L^{2}(\OR)}+\|\partial_{x_1}\Bu\|_{L^{2}(\partial\OR\cap \partial \Omega)})+(\|\nabla\partial_{x_2}\Bu\|_{L^{2}(\OR)}+\|\partial_{x_2}\Bu\|_{L^{2}(\partial\OR\cap \partial \Omega)})\right].
			\end{split}
		\end{equation}
		
		Define
		\[
		\begin{split}
			X(R)=&\int_{\Omega}
			\left(|\nabla\partial_{x_{1}}\Bu|^{2}+|\nabla\partial_{x_{2}}\Bu|^{2}\right)\varphi_{R}\left(\sqrt{x_{1}^{2}+x_{2}^{2}}\right)
			\, d\Bx \\
			&+2\alpha\int_{\partial\Omega}
			\left(|\partial_{x_{1}}\Bu|^{2}+|\partial_{x_{2}}\Bu|^{2}\right)\varphi_{R}\left(\sqrt{x_{1}^{2}+x_{2}^{2}}\right)
			\, dS.
		\end{split}
		\]
		Hence the estimate \eqref{sec6inequ156} implies
		\begin{equation}
			X(R)\leq CR^{\frac{1}{2}}\left[X^{\prime}(R)\right]^{\frac{1}{2}}.
		\end{equation}
		Similarly, it follows from the last part  of the proof for Lemma \ref{Le:32}, one has $\nabla\partial_{x_1}\Bu=\nabla\partial_{x_2}\Bu\equiv 0$ in $\Omega$ and $\partial_{x_1}\Bu=\partial_{x_2}\Bu=0$ on the boundary $\partial\Omega$. Therefore, $\partial_{x_1}\Bu=\partial_{x_2}\Bu\equiv0$. Furthermore, it follows from the divergence free property of $\Bu$ that $\partial_{x_3}u^{3}=0$. Combining with the Navier boundary conditions \eqref{bounNaviersl1} yields that
		\[
		u^{1}=u^{1}(x_{3}), \quad u^{2}=u^{2}(x_{3}) \quad
		\text{and} \quad
		u^{3}\equiv 0.
		\]
		Hence the Navier-Stokes system \eqref{eqsteadyns} reduces to
		\[
		\partial^{2}_{x_3}u^{1}+\partial_{x_1}P=
		\partial^{2}_{x_3}u^{2}+\partial_{x_2}P=
		\partial_{x_3} P=0.
		\]
		This, together with the Navier boundary conditions \eqref{bounNaviersl1}, implies that
		\[
		u^{1}=C_{3}\left(-2\alpha x^{2}_{3}+2\alpha x_{3}+1\right) \,\, \text{and} \,\,
		u^{2}=C_{4}\left(-2\alpha x^{2}_{3}+2\alpha x_{3}+1\right), \quad \text{for some} \,\,\, C_{3}, \, C_{4}\in \mathbb{R}.
		\]
		The proof for case (iii) of Theorem \ref{th:02} is completed.
	\end{proof}	
	\section{Axisymmetric solutions with sublinear growth}\label{Sec6}
	In this section, we give a proof for Theorem \ref{th:03}.
	We first prove a lemma which shows that for $\alpha>0$, the axisymmetric flows in a slab with Navier boundary conditions \eqref{eqNavierbou}
	must be trivial even when the integral $\displaystyle\int_{\Omega_{R}}|\nabla\Bu|^{2}
	\,d\Bx+2\alpha\displaystyle\int_{\partial\Omega_{R}\,\cap\,\partial\Omega}|\Bu|^{2}\, dS$ has the cubic  growth.
	
	\begin{lemma}\label{Le:51h}
		For $\alpha>0$,	let $\Bu$ be a  smooth   solution to the  Navier-Stokes system \eqref{eqsteadyns} in  the slab $\Omega=\mathbb{R}^{2}\times (0,1)$  with Navier boundary conditions \eqref{eqNavierbou}.	   Then $\Bu\equiv 0$ provided that
		$\Bu$ is axisymmetric and
		\begin{equation} \label{growth}
			\varliminf_{R\rightarrow + \infty}  R^{-4}  Z(R) = 0,
		\end{equation}
		where
		\begin{equation}\label{defER}
			\begin{split}
				Z(R)=&\displaystyle \int_0^1 \iint_{\{x_1^2+x_2^2 < R^2\}} |\nabla \Bu(x_1,x_2,x_3)|^2 \, dx_1dx_2dx_3\\
				&+2\alpha\iint_{\{x_1^2+x_2^2 < R^2\}} \left(| \Bu(x_1,x_2,0)|^2 +| \Bu(x_1,x_2,1)|^2 \right)\, dx_1dx_2.
			\end{split}
		\end{equation}
		
	\end{lemma}	
	
	\begin{proof}[Proof of Lemma \ref{Le:51h}]
		The proof contains two steps.
		
		\emph{Step 1.} \emph {Set up.}	Assume that $\Bu$ is a smooth axisymmetric solution to the Navier-Stokes system \eqref{eqsteadyns} in  $\Omega=\mathbb{R}^{2}\times (0,1)$ with Navier boundary conditions \eqref{eqNavierbou} where $\alpha>0$.
		As in the proof  of Lemma \ref{Le:52}, the inequality \eqref{ineqinte} still holds.
		
		Note that
		\begin{equation}\label{eqA29}
			\int_{\Omega}P\Bu \cdot \nabla \varphi_{R} \,d\Bx  =-2\pi \int_{0}^{1}
			\int_{R-1}^{R}Pu^{r}r \, drdz.
		\end{equation}
		The divergence free condition for the axisymmetric solution  is
		\[
		\partial_{r}(ru^{r})+\partial_{z}(ru^{z})=0.
		\]
		Hence for every fixed $r\geq0$, one has
		\begin{equation}\label{eqA30}
			\partial_{r}\int_{0}^{1}r u^{r}\, dz=-\int_{0}^{1}\partial_{z}(ru^{z})\, dz=0,
		\end{equation}
		which implies that
		\begin{equation}\label{eqA31}
			\int_{0}^{1}ru^{r} \, dz
			= \int_0^1 u^r \, dz
			= \int_0^1 \int_{R-1}^R ru^r \, dr dz  =0.
		\end{equation}
		The Poincar\'{e} inequality
		\begin{equation}\label{eqA34}
			\|u^{r}\|_{L^{2}(\OR)}\leq	C\|\partial_{z}u^{r}\|_{L^{2}(\OR)}
		\end{equation}
		holds.
		By virtue of \eqref{eqA31} and Lemma \ref{Bogovskii}, there exists a vector valued function $\Ps_{R}(r,z)\in H^{1}_{0}(D_{R}; \mathbb{R}^{2})$ satisfying
		\begin{equation}\label{eqA35}
			\partial_{r}\Psi_{R}^{r}+\partial_{z}\Psi_{R}^{z}=ru^{r}\ \ \ \ \mbox{in}\ D_R
		\end{equation}
		and
		\begin{equation}\label{eqA36}
			\|\partial_{r}\Ps_{R}\|_{L^{2}(D_{R})}+	\|\partial_{z}\Ps_{R}\|_{L^{2}(D_{R})}
			\leq C\|ru^{r}\|_{L^{2}(D_{R})}
			\leq   CR^{\frac{1}{2}}\|u^{r}\|_{L^{2}(\OR)}.
		\end{equation}
		Therefore, combining \eqref{eqA29} and \eqref{eqA35} one derives
		\begin{equation}\label{eqA37}
			\begin{split}
				\int_{\Omega}P\Bu\cdot \nabla \varphi_{R}\,d\Bx
				=& -2\pi \int_{0}^{1}
				\int_{R-1}^{R}P(\partial_{r}\Psi_{R}^{r}+\partial_{z}\Psi_{R}^{z})\,drdz \\
				=&\ 2 \pi \int_{0}^{1}
				\int_{R-1}^{R}(\partial_{r}P\Psi_{R}^{r}+\partial_{z}P\Psi_{R}^{z})\, drdz.
			\end{split}
		\end{equation}
		Since $\Bu$ is an  axisymmetric  solution of  the   Navier-Stokes system, the gradient of the pressure $(\partial_{r}P, \,  \partial_{z}P)$  satisfies
		\begin{equation}\label{eqA38}
			\left\{
			\begin{aligned}
				&(u^{r}\partial_{r}+u^{z}\partial_{z})u^{r}-\dfrac{(u^{\theta})^2}{r}+\partial_{r}P
				=\left(\partial_{r}^{2}+\dfrac{1}{r}\partial_{r}+\partial_{z}^{2}-\dfrac{1}{r^2}\right)u^{r},\\
				&(u^{r}\partial_{r}+u^{z}\partial_{z})u^{z}+\partial_{z}P
				=\left(\partial_{r}^{2}+\dfrac{1}{r}\partial_{r}+\partial_{z}^{2}\right)u^{z}.
			\end{aligned}
			\right.
		\end{equation}
		According to  \eqref{eqA38},  by integration by parts one obtains 
		\begin{equation}\label{eqA39}
			\begin{split}
				& \int_{0}^{1}\int_{R-1}^{R}\partial_{r}P\Psi^{r}_{R}\, drdz\\
				=& \int_{0}^{1}\int_{R-1}^{R}\left[\left(\partial_{r}^{2}+\dfrac{1}{r}\partial_{r}+\partial_{z}^{2}-\dfrac{1}{r^2}\right)u^{r}\right]\Psi^{r}_{R} \, drdz \\
				&-\int_{0}^{1}\int_{R-1}^{R}\left[(u^{r}\partial_{r}+u^{z}\partial_{z})u^{r}-\dfrac{(u^{\theta})^2}{r}\right]\Psi^{r}_{R}\, drdz	\\
				=&-\int_{0}^{1}\int_{R-1}^{R}(\partial_{r}u^{r}\partial_{r}\Psi_{R}^{r}+\partial_{z}u^{r}\partial_{z}\Psi_{R}^{r}) \, drdz+\int_{0}^{1}\int_{R-1}^{R}\left[\left(\dfrac{1}{r}\partial_{r}-\dfrac{1}{r^2}\right)u^{r}\right]\Psi_{R}^{r}\, drdz\\
				&-\int_{0}^{1}\int_{R-1}^{R}\left[(u^{r}\partial_{r}+u^{z}\partial_{z})u^{r}-\dfrac{(u^{\theta})^2}{r}\right]\Psi^{r}_{R} \, drdz
			\end{split}
		\end{equation}
		and
		\begin{equation}\label{eqA40}
			\begin{split}
				& \int_{0}^{1}\int_{R-1}^{R}\partial_{z}P\Psi^{z}_{R}\, drdz\\
				=&	-\int_{0}^{1}\int_{R-1}^{R}(\partial_{r}u^{z}\partial_{r}\Psi_{R}^{z}+\partial_{z}u^{z}\partial_{z}\Psi_{R}^{z}) \, drdz+\int_{0}^{1}\int_{R-1}^{R}\left(\dfrac{1}{r}\partial_{r}u^{z}\right)\Psi_{R}^{z} \, drdz\\
				&-\int_{0}^{1}\int_{R-1}^{R}\left[\left(u^{r}\partial_{r}+u^{z}\partial_{z}\right)u^{z}\right]\Psi_{R}^{z}\, drdz.
			\end{split}
		\end{equation}
		\emph{Step 2.} \emph {Saint-Venant type estimate.}
		Firstly,  it holds that
		\begin{equation}\label{eqL4uD}
			\|\Bu\|_{L^{4}(D_R)}\leq 	C\|\Bu\|^{\frac{1}{2}}_{L^{2}(D_R)} \|\Bu\|^{\frac{1}{2}}_{H^{1}(D_R)}
			\leq CR^{-\frac{1}{2}}\left(\|\nabla\Bu\|_{L^{2}(\OR)}+\|\Bu\|_{L^{2}(\partial\OR \cap \partial \Omega)}\right).
		\end{equation}
		We now estimate the terms on the right hand side of \eqref{ineqinte}. For the  first two terms, using H\"older inequality, Poincar\'{e}  inequality 	
		\eqref{eqA34} and 	
		\eqref{eqL4uD}, one obtains
		\begin{equation}\label{rese5basces1}
			\begin{split}
				\left|\int_{\Omega}\nabla\varphi_{R}\cdot \nabla\Bu\cdot\Bu  \,d\Bx \right|\leq& C\|\nabla\Bu\|_{L^{2}(\OR)}\cdot \|\Bu\|_{L^{2}(\OR)}\\
				\leq& C\|\nabla\Bu\|_{L^{2}(\OR)}\left(
				\|\nabla\Bu\|_{L^{2}(\OR)}+\|\Bu\|_{L^{2}(\partial\OR\cap \partial \Omega)}
				\right)
			\end{split}
		\end{equation}
		and
		\begin{equation}\label{res5basces2}
			\begin{split}
				\left|\int_{\Omega}\frac{1}{2}\left|\Bu \right|^2\Bu\cdot \nabla \varphi_{R} \,d\Bx\right|&\leq CR\|\Bu\|^{2}_{L^{4}(D_R)}\cdot\|u^{r}\|_{L^{2}(\DR)}  \\
				&\leq CR\cdot R^{-1}\left(\|\nabla\Bu\|_{L^{2}(\OR)}+\|\Bu\|_{L^{2}(\partial\OR\cap \partial \Omega)}\right)^{2}\cdot R^{-\frac{1}{2}}\|\nabla\Bu\|_{L^{2}(\OR)} \\
				&\leq CR^{-\frac{1}{2}}\left(\|\nabla\Bu\|_{L^{2}(\OR)}+\|\Bu\|_{L^{2}(\partial\OR\cap \partial \Omega)}\right)^{3}.
			\end{split}
		\end{equation}
		As for the right hand of \eqref{eqA39}, by \eqref{eqA34}, \eqref{eqA36} 
		and 	
		\eqref{eqL4uD}, one has
		\begin{equation}\label{sec51eqA41}
			\begin{split}
				\left| \int_{0}^{1}\int_{R-1}^{R}(\partial_{r}u^{r}\partial_{r}\Psi_{R}^{r}+\partial_{z}u^{r}\partial_{z}\Psi_{R}^{r}) \, drdz
				\right|\leq & C\|(\partial_{r}, \partial_{z})u^{r}\|_{L^{2}(D_R)}\cdot \|(\partial_{r}, \partial_{z})\Psi_{R}^{r}\|_{L^{2}(D_R)}      \\
				\leq & CR^{-\frac{1}{2}}\|\nabla\Bu\|_{L^{2}(\OR)}\cdot	R^{\frac{1}{2}}\|u^{r}\|_{L^{2}(\OR)}	\\
				\leq&
				C\|\nabla\Bu\|^{2}_{L^{2}(\OR)}
			\end{split}
		\end{equation}
		and
		\begin{equation}\label{sec52eqA42}
			\begin{split}
				\left|\int_{0}^{1}\int_{R-1}^{R}\left[\left(\dfrac{1}{r}\partial_{r}-\dfrac{1}{r^2}\right)u^{r}\right]\Psi_{R}^{r} \, drdz \right|
				\leq&	
				CR^{-1} \cdot R^{-\frac{1}{2}}\| \nabla\Bu\|_{L^{2}(\OR)}
				\cdot R^{\frac{1}{2}} \|u^{r}\|_{L^{2}(\OR)}	\\
				\leq& 
				CR^{-1}\| \nabla\Bu\|^{2}_{L^{2}(\OR)}.
			\end{split}
		\end{equation}
		Furthermore, it holds that
		\begin{equation}\label{sec53eqA43}
			\begin{split}
				&\left|\int_{0}^{1}\int_{R-1}^{R}\left[(u^{r}\partial_{r}+u^{z}\partial_{z})u^{r}\right]\Psi_{R}^{r} \, drdz\right|\\	
				\leq &C\|\Bu\|_{L^{4}(D_R)}\|(\partial_{r}, \partial_{z})u^{r}\|_{L^{2}(D_R)}\cdot \|\Psi_{R}^{r}\|_{L^{4}(D_R)}\\
				\leq
				&CR^{-\frac{1}{2}}\left(\|\nabla\Bu\|_{L^{2}(\OR)}+\|\Bu\|_{L^{2}(\partial\OR\cap \partial \Omega)}\right)\cdot R^{-\frac{1}{2}}\|\nabla \Bu\|_{L^{2}(\OR)}\|(\partial_{r},\partial_{z})\Psi_{R}^{r}\|_{L^{2}(D_R)}\\
				\leq
				&CR^{-\frac{1}{2}}\left(\|\nabla\Bu\|_{L^{2}(\OR)}+\|\Bu\|_{L^{2}(\partial\OR\cap \partial \Omega)}\right)\cdot R^{-\frac{1}{2}}\|\nabla \Bu\|_{L^{2}(\OR)}\cdot R^{\frac{1}{2}}\|\nabla \Bu\|_{L^{2}(\OR)}\\
				\leq &CR^{-\frac{1}{2}}\left(\|\nabla\Bu\|_{L^{2}(\OR)}+\|\Bu\|_{L^{2}(\partial\OR\cap \partial \Omega)}\right)^{3}
			\end{split}
		\end{equation}
		and
		\begin{equation}\label{sec53eqA443}
			\begin{split}
				&\left|\int_{0}^{1}\int_{R-1}^{R}\left[\dfrac{(u^{\theta})^2}{r}\right]\Psi_{R}^{r} \, drdz\right|\\	
				\leq
				&CR^{-1}\|\Bu\|^{2}_{L^{4}(D_R)} \cdot \|\Psi_{R}^{r}\|_{L^{2}(D_R)}\\
				\leq &CR^{-1}\cdot R^{-1}\left(\|\nabla\Bu\|_{L^{2}(\OR)}+\|\Bu\|_{L^{2}(\partial\OR\cap \partial \Omega)}\right)^{2} \cdot R^{\frac{1}{2}}\|\nabla\Bu\|_{L^{2}(\OR)}\\
				\leq &CR^{-\frac{3}{2}}\left(\|\nabla\Bu\|_{L^{2}(\OR)}+\|\Bu\|_{L^{2}(\partial\OR\cap \partial \Omega)}\right)^{3}.
			\end{split}
		\end{equation}
		Collecting the estimates \eqref{sec51eqA41}-\eqref{sec53eqA443} yields
		\begin{equation}\label{eqse5A441}
			\begin{split}
				&\left|\int_{0}^{1}\int_{R-1}^{R}\partial_{r}P\Psi^{r}_{R}\, dr dz \right| \\
				\leq& C\left(\|\nabla\Bu\|_{L^{2}(\OR)}+\|\Bu\|_{L^{2}(\partial\OR\cap \partial \Omega)}\right)^{2}+CR^{-\frac{1}{2}}\left(\|\nabla\Bu\|_{L^{2}(\OR)}+\|\Bu\|_{L^{2}(\partial\OR\cap \partial \Omega)}\right)^{3}.
			\end{split}
		\end{equation}
		Similarly, one has
		\begin{equation}\label{eqsec5A452}
			\begin{split}
				&\left|\int_{0}^{1}\int_{R-1}^{R}\partial_{z}P\Psi^{z}_{R}\, dr dz
				\right|\\
				\leq& C\left(\|\nabla\Bu\|_{L^{2}(\OR)}+\|\Bu\|_{L^{2}(\partial\OR\cap \partial \Omega)}\right)^{2}+CR^{-\frac{1}{2}}\left(\|\nabla\Bu\|_{L^{2}(\OR)}+\|\Bu\|_{L^{2}(\partial\OR\cap \partial \Omega)}\right)^{3}.
			\end{split}
		\end{equation}
		Combining \eqref{rese5basces1}-\eqref{res5basces2} and \eqref{eqse5A441}-\eqref{eqsec5A452}, one arrives at
		\begin{equation}\label{sec512eqBNIii34}
			\begin{split}
				& 	\int_{\Omega}\varphi_{R}|\nabla\Bu|^{2} \,d\Bx+2\alpha\int_{\partial\Omega} \varphi_{R}\left|\Bu \right|^2 \, dS\\
				\leq& C\left(\|\nabla\Bu\|_{L^{2}(\OR)}+\|\Bu\|_{L^{2}(\partial\OR\cap \partial \Omega)}\right)^{2}+CR^{-\frac{1}{2}}\left(\|\nabla\Bu\|_{L^{2}(\OR)}+\|\Bu\|_{L^{2}(\partial\OR\cap \partial \Omega)}\right)^{3}.
			\end{split}
		\end{equation}
		
		Let
		\begin{equation}\label{eqZRinsec5}
			\begin{split}
				Z(R)=&\int_{0}^{1}\iint_{\mathbb{R}^{2}}
				|\nabla\Bu(x_{1}, x_{2}, x_{3})|^{2}\varphi_{R}\left(\sqrt{x_{1}^{2}+x_{2}^{2}}\right)
				\, dx_{1} dx_{2} dx_{3}\\
				&+2\alpha\left(\iint_{\mathbb{R}^{2}}
				|\Bu(x_{1}, x_{2}, 0)|^{2}\varphi_{R}\left(\sqrt{x_{1}^{2}+x_{2}^{2}}\right)
				\, dx_{1} dx_{2}\right.\\
				&+\left.\iint_{\mathbb{R}^{2}}
				|\Bu(x_{1}, x_{2}, 1)|^{2}\varphi_{R}\left(\sqrt{x_{1}^{2}+x_{2}^{2}}\right)
				\, dx_{1} dx_{2} \right).
			\end{split}
		\end{equation}
		Note that for the axisymmetric solution, straightforward computations give
		\begin{equation}
			\begin{split}
				Z(R)=&\, 2\pi\int_{0}^{1}\left(\int_{0}^{R-1}|\nabla\Bu(r,z)|^{2}r\, dr+\int_{R-1}^{R}|\nabla\Bu(r,z)|^{2}(R-r)r\, dr\right )dz\\
				&+4\alpha\pi\left[\left(\int_{0}^{R-1}|\Bu(r,0)|^{2}r\, dr+\int_{R-1}^{R}|\Bu(r,0)|^{2}(R-r)r\, dr\right)\right.\\
				&\left.+\left(\int_{0}^{R-1}|\Bu(r,1)|^{2}r\, dr+\int_{R-1}^{R}|\Bu(r,1)|^{2}(R-r)r\, dr \right)\right]
			\end{split}
		\end{equation}
		and
		\begin{equation}
			Z^{\prime}(R)=\int_{\OR}|\nabla\Bu|^{2}
			\,d\Bx+2\alpha\int_{\partial\OR\cap \partial \Omega}|\Bu|^{2}\, dS.
		\end{equation}
		Hence the estimate \eqref{sec512eqBNIii34} can be written as
		\[
		Z(R)\leq CZ^{\prime}(R)+CR^{-\frac{1}{2}}\left[Z^{\prime}(R)\right]^{\frac{3}{2}}.
		\]
		It follows from  Lemma \ref{le:differineq} (b) that if $Z(R)$ is not identically zero, then
		\[
		\varliminf_{R \rightarrow + \infty} R^{-4} Z(R) > 0.
		\]
		Note that $\displaystyle\varliminf_{R \rightarrow + \infty} R^{-4} Z(R) = 0$. Therefore $Z(R)$ must be identically zero. This implies $\nabla\Bu\equiv 0$ and thus $\Bu\equiv 0$. Hence the proof of Lemma \ref{Le:51h}  is completed.
	\end{proof}
	Now we are ready to prove Theorem \ref{th:03}.
	\begin{proof}[Proof of Theorem \ref{th:03}]
		The proof is almost the same as that for Lemma \ref{Le:51h}. However, we  estimate the terms on the right hand side of  \eqref{ineqinte} in a different way. By   Poincar\'{e} inequality  \eqref{eqA34} and \eqref{eqA39}, one has
		\begin{equation}\label{eqA41}
			\begin{split}
				\left| \int_{0}^{1}\int_{R-1}^{R}(\partial_{r}u^{r}\partial_{r}\Psi_{R}^{r}+\partial_{z}u^{r}\partial_{z}\Psi_{R}^{r}) \, drdz
				\right|
				\leq & CR^{-\frac{1}{2}}\|\nabla\Bu\|_{L^{2}(\OR)}\cdot	R^{\frac{1}{2}}\|u^{r}\|_{L^{2}(\OR)}	\\
				\leq & CR^{-\frac{1}{2}}\|\nabla\Bu\|_{L^{2}(\OR)}\cdot	R^{1}\|\Bu\|_{L^{\infty}(\OR)}	\\
				\leq&
				CR^{\frac{1}{2}}\|\nabla\Bu\|_{L^{2}(\OR)}\|\Bu\|_{L^{\infty}(\OR)}
			\end{split}
		\end{equation}
		and
		\begin{equation}\label{eqA42}
			\begin{split}
				\left|\int_{0}^{1}\int_{R-1}^{R}\left[\left(\dfrac{1}{r}\partial_{r}-\dfrac{1}{r^2}\right)u^{r}\right]\Psi_{R}^{r} \, drdz \right|
				\leq&	
				CR^{-1} \cdot R^{-\frac{1}{2}}\| \nabla\Bu\|_{L^{2}(\OR)}
				\cdot R^{\frac{1}{2}} \|u^{r}\|_{L^{2}(\OR)}  \\
				\leq& CR^{-\frac{3}{2}}\|\nabla\Bu\|_{L^{2}(\OR)}\cdot	R^{1}\|\Bu\|_{L^{\infty}(\OR)}\\
				\leq&
				CR^{-\frac{1}{2}}
				\|\nabla\Bu\|_{L^{2}(\OR)}\|\Bu\|_{L^{\infty}(\OR)}.
			\end{split}
		\end{equation}
		Furthermore, it holds that
		\begin{equation}\label{eqA43}
			\begin{split}
				&\left|\int_{0}^{1}\int_{R-1}^{R}\left[(u^{r}\partial_{r}+u^{z}\partial_{z})u^{r}-\dfrac{(u^{\theta})^2}{r}\right]\Psi_{R}^{r} \, drdz\right|\\	
				\leq
				&C\|\Bu\|_{L^{\infty}(\OR)}\left(R^{-\frac{1}{2}}\|\nabla \Bu\|_{L^{2}(\OR)}+R^{-\frac{3}{2}}\|u^{\theta}\|_{L^{\infty}(\OR)}\right)\cdot R^{\frac{1}{2}}\|u^{r}\|_{L^{2}(\OR)}\\
				\leq &CR^{\frac{1}{2}}\|\nabla\Bu\|_{L^{2}(\OR)}\|\Bu\|^{2}_{L^{\infty}(\OR)}.
			\end{split}
		\end{equation}
		It follows from \eqref{eqA41}-\eqref{eqA43}  that one has
		\begin{equation}\label{sec5eqA45}
			\left|\int_{0}^{1}\int_{R-1}^{R}\partial_{r}P\Psi^{r}_{R}\, drdz
			\right|\leq CR^{\frac{1}{2}}\left(\|\Bu\|_{L^{\infty}(\OR)}+\|\Bu\|^{2}_{L^{\infty}(\OR)}\right)\|\nabla\Bu\|_{L^{2}(\OR)}.
		\end{equation}
		Similarly, one  derives
		\begin{equation}\label{sec5eqA467}
			\left|\int_{0}^{1}\int_{R-1}^{R}\partial_{z}P\Psi^{z}_{R}\, drdz
			\right|\leq CR^{\frac{1}{2}}\left(\|\Bu\|_{L^{\infty}(\OR)}+\|\Bu\|^{2}_{L^{\infty}(\OR)}\right)\|\nabla\Bu\|_{L^{2}(\OR)}.
		\end{equation}
		Collecting \eqref{basces1}-\eqref{basces2} and \eqref{sec5eqA45}-\eqref{sec5eqA467} gives
		\begin{equation}\label{ineqse5zR}
			Z(R)\leq CR^{\frac{1}{2}}\left(\|\Bu\|_{L^{\infty}(\OR)}+\|\Bu\|^{2}_{L^{\infty}(\OR)}\right) \left[Z^{\prime}(R)\right]^{\frac{1}{2}},
		\end{equation}
		where $Z(R)$  is defined in \eqref{eqZRinsec5}.
		
		Suppose $\Bu$ is not identically equal to zero and $\Bu$ satisfies \eqref{growth-coro1}.  For any small $\epsilon>0$, there exists a constant $R_{0}(\epsilon)>2$ such that
		\[
		\|\Bu\|_{L^{\infty}(\OR)}\leq \epsilon R \quad \text{for any} \, R\geq R_{0}(\epsilon).
		\]
		Due to $Z(R)>0$, the inequality \eqref{ineqse5zR} implies that
		\begin{equation}\label{eqinse5}
			(C\epsilon)^{-2}R^{-5}\leq \frac{Z^{\prime}(R)}{\left[Z(R)\right]^{2}}.
		\end{equation}
		If $\Bu$ is not equal to zero, according to Lemma \ref{Le:51h}, $Z(R)$ must be unbounded as $R\rightarrow +\infty$. For $R$ sufficiently large, integrating
		\eqref{eqinse5} over $[R, +\infty)$ one arrives at
		\begin{equation}\label{eqinse561}
			R^{-4}Z(R)\leq 4(C\epsilon)^{2}.
		\end{equation}
		Since $\epsilon$ can be arbitrarily  small, this implies
		\eqref{growth} and leads to a contradiction with the assumption that $\Bu$ is not identically zero. This finishes the proof of Theorem \ref{th:03}.
	\end{proof}


	
	



	
	\begin{appendices}
		\section*{Appendix: Regularity of bounded solutions}\label{secApp1}
		In this appendix, we give the proof of regularity estimates for Navier-Stokes system with Navier boundary conditions in detail. It might be useful for people who are interested with this estimate.
		\begin{proof}[Proof of Lemma \ref{Le:NAvbouns}]
			The proof is more or less standard. It consists of  interior regularity and boundary regularity estimates.
			
			\emph{Case 1. Boundary regularity.}
			Instead of the steady Navier-Stokes system \eqref{eqsteadyns}, we first consider the following Stokes system in a slab $ \Omega=\mathbb{R}^{2}\times (0,1)$ with Navier boundary conditions
			\begin{equation}\label{eqstokespb}
				\begin{cases}
					-\Delta \Bu +\nabla P=\Bf+\text{div} \, \mathbb{F},\,\,\,\,\, \nabla \cdot \Bu=0, \
					& \quad {\text{in} \ \Omega}, \\
					\Bu\cdot \Bn=0,\,\,\,\,\, \left[\Bn\cdot\left(2\D(\Bu)+\mathbb{F}\right)+2\alpha\Bu\right]\cdot \Bt=0, \
					& \quad  {\text{on} \ \partial \Omega},
				\end{cases}
			\end{equation}
			where	the friction coeffcient $\alpha$ satisfies
			$\alpha>0$,  $\Bf=(f^{1}, f^{2}, f^{3})$ and $\mathbb{F}$ is a $3\times3$ matrix. Here we define the domain $\widehat{\mathscr{B}_{r}}(\Bx)=\mathscr{B}_{r}(\Bx)\cap \Omega$,  $\Bx\in \partial\Omega$.
			
			
			Inspired by \cite[Section 4]{TACGJDE21} and \cite[Section 2.4]{GTH}, i.e., the $L^{2}$ theory of Stokes system, we establish the $H^{3}$ boundary estimate for  the Stokes problem \eqref{eqstokespb}.
			
			\emph{Step 1.} \emph {$H^{1}$ boundary estimate.}	Let $\xi\in C^{\infty}_{c}(\mathbb{R}^{3})$ be a smooth cut-off function satisfying 
			\begin{equation*}\label{smoothcutfun}
				\begin{split}
					&\xi\equiv 1,\,\,\,\, \text{on}\,\, \mathscr{B}_{\frac{3}{4}}(\Bx),\quad  \xi\equiv 0,\,\,\,\, \text{on}\,\, \mathbb{R}^{3} \backslash  \mathscr{B}_{\frac{5}{6}}(\Bx),
					\quad	0\leq \xi \leq 1,
				\end{split}
			\end{equation*}	
			and $\supp(\xi) \cap \mathbb{R}^3_{+}$ is smooth. Later on, the constant $C$ may depend on the cut-off function $\xi$.
			Assume that $\Bf\in L^{2}\left(\widehat{\mathscr{B}_{1}}\right)$ and $\mathbb{F}\in \mathbb{L}^{2}\left(\widehat{\mathscr{B}_{1}}\right)$,  multiplying the  equation by test function $\xi^{2}\Bu$, and integrating over $\widehat{\mathscr{B}_{1}}$, one has
			\begin{equation*}
				\int_{\widehat{\mathscr{B}_{1}}} (-\Delta \Bu + \nabla P) \cdot \xi^2\Bu \, d\Bx = \int_{\widehat{\mathscr{B}_{1}}} (\Bf + {\rm div} \mathbb{F} )  \cdot \xi^2 \Bu \, d\Bx.
			\end{equation*}
			Integration by parts and using Navier boundary conditions yield 
			\[
			\begin{split}
				&\int_{\widehat{\mathscr{B}_{1}}}-\Delta \Bu \cdot \xi^{2}\Bu\, d\Bx
				=\int_{\widehat{\mathscr{B}_{1}}}-2 \div\D(\Bu) \cdot \xi^{2}\Bu\, d\Bx\\
				=&\int_{\widehat{\mathscr{B}_{1}}}2\D(\Bu) :\D(\xi^{2}\Bu) \, d\Bx-\int_{\partial\widehat{\mathscr{B}_{1}}}2\Bn\cdot \D(\Bu)\cdot \xi^{2}\Bu\, dS\\
				=&\int_{\whBo}2\D(\Bu) :(\nabla\xi^{2}\otimes\Bu) \, d\Bx+\int_{\widehat{\mathscr{B}_{1}}}2\xi^{2}\D(\Bu) :\D(\Bu) \, d\Bx\\
				&+\int_{\partial\widehat{\mathscr{B}_{1}}}2\xi^{2}\alpha |\Bu|^{2}\, dS+\int_{\partial\widehat{\mathscr{B}_{1}}}\xi^{2}\Bn\cdot\mathbb{F}\cdot \Bu\, dS,	
			\end{split}
			\]
			where $\nabla\xi^{2}\otimes\Bu$ is the  matrix with components $\partial_{x_i}\xi^{2}u_{j}$.
			On the other hand,
			\[
			\begin{split}
				\int_{\whBo}2\D(\xi\Bu) :\D(\xi\Bu) \, d\Bx=&\int_{\whBo}2\D(\Bu) :(\nabla\xi^{2}\otimes\Bu) \, d\Bx+\int_{\whBo}2\xi^{2}\D(\Bu) :\D(\Bu) \, d\Bx\\
				&\,\,+\int_{\whBo}|\Bu|^{2}|\nabla \xi|^{2}\, d\Bx+\int_{\whBo}[(\nabla \xi)\otimes(\nabla \xi)]:(\Bu\otimes \Bu)\, d\Bx.
			\end{split}
			\]	
			Hence one has
			\begin{equation}\label{lapeq1hh}
				\begin{split}
					&\int_{\whBo}-\Delta \Bu \cdot \xi^{2}\Bu\, d\Bx\\
					=&\int_{\whBo}2\D(\xi\Bu) :\D(\xi\Bu) \, d\Bx-\int_{\whBo}|\Bu|^{2}|\nabla \xi|^{2}\, d\Bx-\int_{\whBo}[(\nabla \xi)\otimes(\nabla \xi)]:(\Bu\otimes \Bu)\, d\Bx\\
					&\,\,+\int_{\partial\whBo}2\alpha |\xi\Bu|^{2}\, dS+\int_{\partial\whBo}\xi^{2}\Bn\cdot\mathbb{F}\cdot \Bu\, dS.
				\end{split}
			\end{equation}
			For other terms, one has
			\begin{equation}\label{eqpreess12}
				\int_{\whBo} \nabla P\cdot \xi^{2}\Bu\,d\Bx
				=-\int_{\whBo}2P\xi\nabla \xi\cdot \Bu\,d\Bx
			\end{equation}
			and
			\begin{equation}\label{eqFes12}
				\int_{\whBo}\div \mathbb{F} \cdot \xi^{2}\Bu\,d\Bx
				=-\int_{\whBo}\left[\xi\mathbb{F}:\nabla(\xi\Bu)+\mathbb{F}:(\xi\nabla\xi\otimes\Bu)\right]\,d\Bx+\int_{\partial\whBo}\xi^{2}\Bn\cdot\mathbb{F}\cdot \Bu\, dS.
			\end{equation}
			Collecting \eqref{lapeq1hh}-\eqref{eqFes12}, one derives
			\begin{equation}\label{eqDutrans1}
				\begin{split}
					&\int_{\whBo} 2\D(\xi\Bu): \D(\xi \Bu) \,d\Bx+\int_{\partial\whBo}2\alpha |\xi\Bu|^{2}\, dS\\
					=&\int_{\whBo} \Bf\cdot \xi^{2} \Bu \,d\Bx+\int_{\whBo}|\Bu|^{2}|\nabla \xi|^{2}\, d\Bx+\int_{\whBo}[(\nabla \xi)\otimes(\nabla \xi)]:(\Bu\otimes \Bu)\, d\Bx\\
					&\,\,-\int_{\whBo}\left[\xi\mathbb{F}:\nabla (\xi\Bu)+\mathbb{F}:(\xi\nabla\xi\otimes\Bu)\right]\,d\Bx+\int_{\whBo}2P\xi\nabla \xi\cdot \Bu\,d\Bx.
				\end{split}
			\end{equation}	 		
			Note that $\whBo$ is a bounded Lipschitz, axisymmetric domain  and $\xi\Bu\cdot\Bn=0$ on $\partial \whBo$. From \cite[Proposition 3.13]{TACGJDE21} and \cite[Section 2.3.11]{GTH}, it holds that
			\begin{equation}\label{equival12}
				\begin{split}
					&\, C^{-1}\left( \|\D(\xi\Bu)\|_{\mathbb{L}^{2}\left(\whBo\right)}+\|\sqrt{\alpha}\xi\Bu\cdot \Bt\|_{L^{2}\left(\partial\whBo\right)}\right)\\
					\leq&\|\xi\Bu\|_{H^{1}\left(\whBo\right)}\leq C\left( \|\D(\xi\Bu)\|_{\mathbb{L}^{2}\left(\whBo\right)}+\|\sqrt{\alpha}\xi\Bu\cdot \Bt\|_{L^{2}\left(\partial\whBo\right)}\right),
				\end{split}
			\end{equation}	
			for a constant $C=C(\alpha)$.
			
			As for the term $\displaystyle\int_{\whBo} 2P\xi\nabla \xi\cdot \Bu\,d\Bx$, let $\witBo$ be a smooth domain satisfying $\widehat{\mathscr{B}_{\frac{5}{6}}} \subset \witBo \subset \whBo $. Denote  \begin{displaymath}
				\xi \nabla \xi\cdot\Bu=\psi 
				\quad
				\text{and}
				\quad
				\bar{\psi}=\dfrac{1}{|\witBo|}\displaystyle\int_{\witBo}\psi\, d\Bx.
			\end{displaymath} 
			Assume that $\displaystyle \int_{\witBo} P \, d\Bx  = 0,$ otherwise, consider
			$P - \dfrac{1}{ |\witBo|} \displaystyle \int_{\witBo} P \, d\Bx $ instead. It holds that
			\[
			\int_{\whBo}P\psi\,d\Bx = \int_{\witBo} P \psi \, d\Bx  =  \int_{\witBo}P(\psi-\bar{\psi})\,d\Bx.
			\]
			Note that $\displaystyle\int_{\witBo}(\psi-\bar{\psi})\,d\Bx=0$ and  $\psi-\bar{\psi}\in W^{1,2}(\witBo;\mathbb{R}^{3})$.
			By  Lemma \ref{Bogovskii}, there is a $\boldsymbol{\Phi}\in W^{2,2}(\witBo;\mathbb{R}^{3})$ with $\boldsymbol{\Phi}|_{\partial\witBo}=0$  such that 
			\begin{displaymath}
				\div\, \boldsymbol{\Phi}=\psi-\bar{\psi}, \quad \text{in} \,\,\,
				\witBo.
			\end{displaymath}
			Using the Stokes system \eqref{eqstokespb} and integration by parts give
			\[
			\begin{split}
				\int_{\whBo}P\psi\,d\Bx
				&= 	\int_{\witBo}P\div\, \boldsymbol{\Phi}\,d\Bx= 	\int_{\witBo}-\nabla P\cdot \boldsymbol{\Phi}\,d\Bx\\
				&=-\int_{\witBo}(\Delta \Bu+\Bf+\div \mathbb{F})\cdot \boldsymbol{\Phi}\,d\Bx\\
				&=\int_{\witBo} -\Bu\cdot \Delta \boldsymbol{\Phi}-\Bf\cdot \boldsymbol{\Phi}+ (\mathbb{F}\otimes\nabla\boldsymbol{\Phi})\,d\Bx+ \int_{\partial\witBo}\Bu\frac{\partial \boldsymbol{\Phi}}{\partial\Bn}\, dS.
			\end{split}
			\]
			By the  trace theorem, one has
			\[
			\|\boldsymbol{\Phi}\|_{H^{1}\left(\partial\witBo\right)}\leq C(\Omega) \|\boldsymbol{\Phi}\|_{H^{\frac{3}{2}}\left(\witBo\right)}.
			\]
			Due to the $H^{2}$ estimate for $\boldsymbol{\Phi}$, one derives
			\[
			\|\boldsymbol{\Phi}\|_{H^{2}\left(\witBo\right)}
			\leq C(\Omega)\|\psi\|_{H^{1}\left(\witBo\right)}\leq C(\Omega)\|\xi \Bu\|_{H^{1}\left(\whBo\right)}.
			\]
			Combining these estimates and due to that $\Bu$ is bounded on $\partial \witBo$, one derives
			\begin{equation}\label{preestima234}
				\begin{split}
					&\left|\displaystyle\int_{\whBo}2P\xi\nabla \xi\cdot \Bu\,d\Bx \right|\\
					\leq&
					C(\Omega)\left[\|\xi\Bu\|_{H^{1}\left(\whBo\right)}\left(\|\mathbb{F}\|_{\mathbb{L}^{2}\left(\whBo\right)}+\|\Bf\|_{L^{2}\left(\whBo\right)}+\|\Bu\|_{L^{2}\left(\whBo\right)}+\|\Bu\|_{L^{\infty}\left(\partial \witBo\right)}\right)\right].
				\end{split}
			\end{equation}
			From \eqref{eqDutrans1}-\eqref{preestima234} and Cauchy's inequality,
			it holds that
			\begin{equation}\label{W1PESSTOKS1}
				\begin{split}
					\|\Bu\|_{H^{1}\left(\widehat{\mathscr{B}_{\frac{3}{4}}}\right)}+	\|P\|_{L^{2}\left(\widehat{\mathscr{B}_{\frac{3}{4}}}\right)}
					\leq   C(\Omega, \alpha)
					\left(\|\mathbb{F}\|_{\mathbb{L}^{2}\left(\whBo\right)}+\|\Bf\|_{L^{2}\left(\whBo\right)}+\|\Bu\|_{L^{2}\left(\whBo\right)}+\|\Bu\|_{L^{\infty}(\Omega)}\right).
				\end{split}
			\end{equation}
			\emph{Step 2.} \emph {$H^{2}$ boundary estimate.} The proof is almost the same as that in \cite[Theorem 4.5]{TACGJDE21} and \cite[Theorem 2.4.5]{GTH}. We write it down for readers' convenience.  We consider only the case that $\mathbb{F}=0$.   Denote the difference quotient by
			\[
			D_{k}^{h}\Bu(\Bx)=\frac{\Bu(\Bx+h\Be_{k})-\Bu(\Bx)}{h}, \quad k= 1, 2,  \quad h\in \mathbb{R}.
			\]
			We choose  a smooth cut-off function $\eta\in C^{\infty}_{c}(\mathbb{R}^{3})$ such that
			\begin{equation*}\label{smoothcutfun}
				\eta\equiv 1,\,\,\,\, \text{on}\,\, \mathscr{B}_{\frac{1}{2}}(\Bx),\quad \eta\equiv 0,\,\,\,\, \text{on}\,\, \mathbb{R}^{3} \backslash  \mathscr{B}_{\frac{3}{4}}(\Bx)
				\, \, \, \text{and} \, \, \,	0\leq \eta \leq 1.
			\end{equation*}	
			\emph{(i)} \emph {Tangential $H^{2}$ estimate of the velocity.}
			Let $h>0$ be small and $\BPsi=-D_{k}^{-h}(\eta^{2}D_{k}^{h}\Bu)$, with $k=1,2$. Taking $\BPsi$ as a test function, one obtains
			\begin{equation}\label{variafortes2}
				\begin{split}
					&2\int_{\sfwitBo}\eta^{2}|D^{h}_{k}\D\Bu|^{2}\,d\Bx+2\int_{\sfwitBo}D^{h}_{k}\D\Bu:(2\eta\nabla\eta\otimes D^{h}_{k}\Bu)\,d\Bx+\int_{\partial\sfwitBo}2\alpha \eta^{2}|D^{h}_{k}\Bu|^{2}\, dS\\
					&\,\,-\int_{\sfwitBo}P\div(-D_{k}^{-h}(\eta^{2}D_{k}^{h}\Bu))\, d\Bx=\int_{\sfwitBo} \Bf\cdot (-D_{k}^{-h}(\eta^{2}D_{k}^{h}\Bu))\,d\Bx.
				\end{split}
			\end{equation}
			Note that there exists an  $\epsilon>0$ such that
			\begin{equation}\label{fireqesu1}
				\left|	\int_{\sfwitBo}D^{h}_{k}\D\Bu:(2\eta\nabla\eta\otimes D^{h}_{k}\Bu)\,d\Bx\right|\leq C\left[\epsilon\int_{\sfwitBo}\eta^{2}|D^{h}_{k}\D\Bu|^{2}\,d\Bx+\frac{1}{\epsilon}\int_{\sfwitBo}|D^{h}_{k}\Bu|^{2}\,d\Bx\right].
			\end{equation}
			Similarly, one has
			\begin{equation}\label{eqesP2}
				\begin{split}
					& \left|\int_{\sfwitBo}P\div(-D_{k}^{-h}(\eta^{2}D_{k}^{h}\Bu))\, d\Bx\right|\\
					\leq&\epsilon\int_{\sfwitBo}|\div(-D_{k}^{-h}(\eta^{2}D_{k}^{h}\Bu))|^{2}\, d\Bx+\frac{C}{\epsilon}\int_{\sfwitBo}|P|^{2}\,d\Bx\\
					\leq& C\left[\epsilon\left(\int_{\sfwitBo}|D^{h}_{k}\Bu|^{2}\,d\Bx +\int_{\sfwitBo}\eta^{2}|D^{-h}_{k}D^{h}_{k}\Bu|^{2}\,d\Bx\right)\right]+\frac{C}{\epsilon}\int_{\sfwitBo}|P|^{2}\,d\Bx\\
					\leq& C\left[\epsilon\left(\int_{\sfwitBo}|D^{h}_{k}\Bu|^{2}\,d\Bx +\int_{\sfwitBo}\eta^{2}|\nabla D^{h}_{k}\Bu|^{2}\,d\Bx\right)\right]+\frac{C}{\epsilon}\int_{\sfwitBo}|P|^{2}\,d\Bx,
				\end{split}
			\end{equation}
			where the second inequality is due to
			\[
			\begin{split}
				\div(D_{k}^{-h}\left(\eta^{2}D_{k}^{h}\Bu\right))
				=&D_{k}^{-h}\div(\eta^{2}D_{k}^{h}\Bu)=D_{k}^{-h}
				(2\eta\nabla\eta\cdot D_{k}^{h}\Bu)+D_{k}^{-h}\left(\eta^{2}\underbrace{\div (D_{k}^{h}\Bu)}_{=0}\right)\\
				=&D_{k}^{-h}(2\eta\nabla\eta)\cdot D_{k}^{h}\Bu(\Bx-h\Be_{k})+2\eta\nabla\eta\cdot D_{k}^{-h} D_{k}^{h}\Bu.
			\end{split}
			\]	
			Similarly, one has
			\begin{equation}\label{eqfes3}
				\begin{split}
					& \left|	\int_{\sfwitBo}\Bf\cdot (-D_{k}^{-h}(\eta^{2}D_{k}^{h}\Bu))\, d\Bx\right|\\
					\leq&
					\epsilon\int_{\sfwitBo}|(D_{k}^{-h}(\eta^{2}D_{k}^{h}\Bu))|^{2}\, d\Bx+\frac{C}{\epsilon}\int_{\sfwitBo}|\Bf|^{2}\,d\Bx\\
					\leq&
					\epsilon\int_{\sfwitBo}|(\nabla(\eta^{2}D_{k}^{h}\Bu))|^{2}\,d\Bx+\frac{C}{\epsilon}\int_{\sfwitBo}|\Bf|^{2}\,d\Bx\\
					\leq& C\left[\epsilon\left(\int_{\sfwitBo}|D^{h}_{k}\Bu|^{2}\,d\Bx +\int_{\sfwitBo}\eta^{2}|\nabla D^{h}_{k}\Bu|^{2}\,d\Bx\right)\right]+\frac{C}{\epsilon}\int_{\sfwitBo}|\Bf|^{2}\,d\Bx.
				\end{split}
			\end{equation} 	
			For $\alpha>0$, it follows from \eqref{variafortes2}-\eqref{eqfes3}  that
			\begin{equation}\label{eq12eta}
				\begin{split}
					&2\int_{\sfwitBo}\eta^{2}|D^{h}_{k}\D\Bu|^{2}\,d\Bx\\
					\leq&C \epsilon\left(\int_{\sfwitBo}\eta^{2}|\D D^{h}_{k}\Bu|^{2}\,d\Bx+\int_{\sfwitBo}\eta^{2}|\nabla D^{h}_{k}\Bu|^{2}\,d\Bx\right)+\frac{C}{\epsilon}\left(\int_{\sfwitBo}|\Bf|^{2}\,d\Bx+\int_{\sfwitBo}|P|^{2}\,d\Bx\right)\\
					&\,\,+C\int_{\sfwitBo}|D^{h}_{k}\Bu|^{2}\,d\Bx\\
					\leq&C \epsilon\left(\int_{\sfwitBo}\eta^{2}|\nabla D^{h}_{k}\Bu|^{2}\,d\Bx\right)+\frac{C}{\epsilon}\left(\int_{\sfwitBo}|\Bf|^{2}\,d\Bx+\int_{\sfwitBo}|P|^{2}\,d\Bx\right)+C\int_{\sfwitBo}|D^{h}_{k}\Bu|^{2}\,d\Bx.
				\end{split}
			\end{equation}	
			Note that
			\begin{equation}
				\begin{split}
					\|\eta\nabla D_{k}^{h}\Bu\|^{2}_{\mathbb{L}^{2}\left(\sfwitBo\right)}
					&=\|\nabla (\eta D_{k}^{h}\Bu)-(\nabla \eta\otimes D_{k}^{h}\Bu)\|^{2}_{\mathbb{L}^{2}\left(\sfwitBo\right)} \\
					&\leq	C\left(\|\eta D_{k}^{h}\Bu\|_{H^{1}\left(\sfwitBo\right)}^{2}+\| D_{k}^{h}\Bu\|_{L^{2}\left(\sfwitBo\right)}^{2} \right).
				\end{split}
			\end{equation}
			Furthermore, by  Korn's inequality \eqref{kornin1} one has
			\begin{equation}\label{korinu2}
				\begin{split}
					\|\eta D_{k}^{h}\Bu\|_{H^{1}\left(\sfwitBo\right)}
					\leq& C\left(\|\eta D_{k}^{h}\Bu\|_{L^{2}\left(\sfwitBo\right)}+\|\D(\eta D_{k}^{h}\Bu)\|_{\mathbb{L}^{2}\left(\sfwitBo\right)}\right)\\
					\leq& C\left(\|\eta D_{k}^{h}\Bu\|_{L^{2}\left(\sfwitBo\right)}+\|\nabla\eta
					D_{k}^{h}\Bu\|_{L^{2}\left(\sfwitBo\right)}+\|\eta\D D_{k}^{h}\Bu\|_{\mathbb{L}^{2}\left(\sfwitBo\right)}\right)\\
					\leq& C\left(\| D_{k}^{h}\Bu\|_{L^{2}\left(\sfwitBo\right)}+\|\eta\D D_{k}^{h}\Bu\|_{\mathbb{L}^{2}\left(\sfwitBo\right)}\right).
				\end{split}
			\end{equation}
			Combining \eqref{eq12eta}-\eqref{korinu2} and choosing  $\epsilon$ sufficiently small, one has
			\begin{equation}\label{eqH1est1}
				\|D_{k}^{h}\Bu\|^{2}_{H^{1}\left(\widehat{\mathscr{B}_{\frac{1}{2}}}\right)}\leq 
				\|\eta D_{k}^{h}\Bu\|^{2}_{H^{1}\left(\sfwitBo\right)}
				\leq C\left(\|\Bf\|^{2}_{L^{2}\left(\sfwitBo\right)}+\|P\|^{2}_{L^{2}\left(\sfwitBo\right)}+\|D^{h}_{k}\Bu\|^{2}_{L^{2}\left(\sfwitBo\right)}\right).
			\end{equation}
			By adding a constant, we can assume that $P\in L^2_0\left(\sfwitBo\right)$. Hence $P$ satisfies
			\begin{equation}\label{eqPes176}
				\begin{split}
					&\|P\|_{L^{2}\left(\sfwitBo\right)}\leq C(\Omega)\|\nabla P\|_{H^{-1}\left(\sfwitBo\right)}\\
					\leq&
					C(\Omega)\|\Delta\Bu+\Bf\|_{H^{-1}\left(\sfwitBo\right)}\\
					\leq& C(\Omega)\left(\|\Bf\|_{L^{2}\left(\sfwitBo\right)}+\|\Bu\|_{H^{1}\left(\sfwitBo\right)}\right).
				\end{split}
			\end{equation}
			It follows from \eqref{eqH1est1} and \eqref{eqPes176} that one has
			\begin{equation}\label{es12uij}
				\left\|\frac{\partial^{2}\Bu}{\partial{x_i}\partial{x_j}}\right
				\|_{L^{2}\left(\widehat{\mathscr{B}_{\frac{1}{2}}}\right)}\leq C\left(\|\Bf\|_{L^{2}\left(\sfwitBo\right)}+\|\Bu\|_{H^{1}\left(\sfwitBo\right)}\right),
			\end{equation}
			for all $i,j=1, 2, 3$ except $i=j=3$.
			
			\emph{(ii)} \emph {Normal  $H^{2}$ estimate of the velocity.}
			We first deduce the tangential  regularity of the pressure. It is noted that
			\[
			\nabla \frac{\partial P}{\partial x_{i}}
			=\frac{\partial }{\partial x_{i}}(\nabla P)=\frac{\partial }{\partial x_{i}}(\Bf+\Delta \Bu)
			=\frac{\partial \Bf}{\partial x_{i}}+\div \left(\nabla \frac{\partial \Bu}{\partial x_{i}}\right),
			\]
			for $i=1,2$. Since there is  no term of the form $\partial^{2}\Bu/\partial{x_{3}^{2}}$, combining the estimate  \eqref{es12uij},
			we obtain $ \nabla \frac{\partial P}{\partial x_{i}}\in H^{-1}\left(\widehat{\mathscr{B}_{\frac{1}{2}}}\right)$. In fact, it follows from Ne\v{c}as inequality \eqref{necas1in} and \eqref{eqPes176}-\eqref{es12uij} that
			\begin{equation}\label{estiPpar17}
				\begin{split}
					\left\|\frac{\partial P}{\partial x_{i}}\right\|_{L^{2}\left(\widehat{\mathscr{B}_{\frac{1}{2}}} \right)} 
					&\leq C\left(\left\|\nabla\frac{\partial P}{\partial x_{i}}\right\|_{H^{-1}\left(\widehat{\mathscr{B}_{\frac{1}{2}}}\right)}+\left\|\frac{\partial P}{\partial x_{i}}\right\|_{H^{-1}\left(\widehat{\mathscr{B}_{\frac{1}{2}}}\right)}\right)\\	
					&\leq C\left(\|\Bf\|_{L^{2}\left(\sfwitBo\right)}+\|\Bu\|_{H^{1}\left(\sfwitBo\right)}\right).
				\end{split}
			\end{equation}
			
			Next we study the normal regularity of the velocity $\Bu$ and the pressure $P$.
			Differentiating the divergence free equation with respect to $x_{3}$, we get
			\[
			\frac{\partial^{2}u^{3}}{\partial x_{3}^{2}}=-\sum_{i=1}^{2}\frac{\partial^{2}u^{i}}{\partial x_{i}\partial x_{3}}\in L^{2}\left(\widehat{\mathscr{B}_{\frac{1}{2}}}\right) \quad \text{and} \quad
			\frac{\partial P}{\partial x_{3}}=f^{3}+\Delta u^{3} \in L^{2}\left(\widehat{\mathscr{B}_{\frac{1}{2}}}\right).
			\]
			For $i=1, 2$,
			\[
			\frac{\partial^{2}u^{i}}{\partial{x_{3}^{2}}}=-\sum_{k=1}^{2}\frac{\partial^{2}u^{i}}{\partial{x_{k}^{2}}}-f^{i}+\frac{\partial P}{\partial x_{i}}\in L^{2}\left(\widehat{\mathscr{B}_{\frac{1}{2}}}\right).
			\]
			These identities together with the estimates \eqref{es12uij}-\eqref{estiPpar17} imply that
			\begin{equation}\label{W1PESSTOKS2}
				\|\Bu\|_{H^{2}\left(\widehat{\mathscr{B}_{\frac{1}{2}}}\right)}+\|P\|_{H^{1}\left(\widehat{\mathscr{B}_{\frac{1}{2}}}\right)}\leq C(\Omega,\alpha)\left(\|\Bf\|_{L^{2}\left(\sfwitBo\right)}+\| \Bu\|_{H^{1}\left(\sfwitBo\right)}\right).
			\end{equation}
			
			\emph{Step 3.} \emph {$H^{3}$ boundary estimate.}	Here we consider the case $\Bf\in H^{1}\left(\widehat{\mathscr{B}_{\frac{1}{2}}}\right)$ and $\mathbb{F} \equiv 0$. By the interior regularity result \cite[Theorem IV.4.1]{GAGP11}, $\Bu\in H^{3}\left(\widehat{\mathscr{B}_{\frac{1}{2}}}\right)$. For $i=1, 2$, taking the derivatives with respect to $x_i$,
			\begin{equation}\label{diffeqstokespb}
				\begin{cases}
					-\Delta \partial_{x_i} \Bu +\nabla \partial_{x_i}P=\partial_{x_i}\Bf,\,\,\,\,\, \nabla \cdot \partial_{x_i} \Bu=0, \
					& \quad {\text{in} \ \widehat{\mathscr{B}_{\frac{1}{2}}}}, \\
					\partial_{x_i}\Bu\cdot \Bn=0,\,\,\,\,\, \left[\Bn\cdot\D(\partial_{x_i}\Bu)+\alpha \partial_{x_i}\Bu\right]\cdot \Bt=0, \
					& \quad  {\text{on} \ \widehat{\mathscr{B}_{\frac{1}{2}}} \cap \partial \Omega}.
				\end{cases}
			\end{equation}
			$\partial_{x_i} \Bu$ satisfies the
			$H^{2}$ estimate, i.e.,
			\begin{equation}\label{W1PESSTOKS23ne}
				\|\partial_{x_i}\Bu\|_{H^{2}\left(\widehat{\mathscr{B}_{\frac{1}{4}}}\right)}+\|\partial_{x_i}P\|_{H^{1}\left(\widehat{\mathscr{B}_{\frac{1}{4}}}\right)}\leq C(\Omega,\alpha)\left(\|\Bf\|_{H^{1}\left(\widehat{\mathscr{B}_{\frac{1}{2}}}\right)}+\| \Bu\|_{H^{2}\left(\widehat{\mathscr{B}_{\frac{1}{2}}}\right)}\right).
			\end{equation}
			Finally,  taking the $\partial_{x_3}^2$-derivative of divergence free equation 
			\begin{equation*}
				\frac{\partial^{3}u^{3}}{\partial x_{3}^{3}}=-\sum_{k=1}^{2}\frac{\partial^{3}u^{k}}{\partial x_{3}^{2}\partial x_{k}}\in L^{2}\left(\widehat{\mathscr{B}_{\frac{1}{4}}}\right).
			\end{equation*}
			Differentiating the Stokes system \eqref{eqstokespb} with respect to $x_{3}$, one gets
			\begin{equation*}
				\frac{\partial^{3}u^{i}}{\partial x_{3}^{3}}=-\sum_{k=1}^{2}\frac{\partial^{3}u^{i}}{\partial x_{3}\partial x^{2}_{k}}+\partial_{x_3}\partial_{x_i} P+\partial_{x_3}f^{i} \in L^{2}\left(\widehat{\mathscr{B}_{\frac{1}{4}}}\right)
			\end{equation*}
			and
			\begin{equation*}
				\partial^{2}_{x_3}P=\Delta\partial_{x_3}u^{3}+\partial_{x_3}f^{3} \in L^{2}\left(\widehat{\mathscr{B}_{\frac{1}{4}}}\right).
			\end{equation*}
			Hence,  the $H^{3}$ boundary estimate holds,
			\begin{equation}\label{W1PESSTOKS456}
				\|\Bu\|_{H^{3}\left(\widehat{\mathscr{B}_{\frac{1}{4}}}\right)}+\|P\|_{H^{2}\left(\widehat{\mathscr{B}_{\frac{1}{4}}}\right)}\leq C(\Omega,\alpha)\left(\|\Bf\|_{H^{1}\left(\widehat{\mathscr{B}_{\frac{1}{2}}}\right)}+\|\Bu\|_{H^{2}\left(\widehat{\mathscr{B}_{\frac{1}{2}}}\right)}\right).
			\end{equation}
			
			Now we turn to the regularity estimates for bounded solutions $\Bu$ to the Navier-Stokes system \eqref{eqsteadyns}.	
			The $H^{1}$ estimate 	\eqref{W1PESSTOKS1} tells that
			\begin{equation}
				\|\nabla\Bu\|_{L^{2}\left(
					\mathscr{B}_{\frac{3}{4}}(\Bx)\cap \Omega\right)}
				\leq C
				\left(\|\Bu\otimes\Bu\|_{\mathbb{L}^{2} \left(\mathscr{B}_{1}
					(\Bx)\cap \Omega\right)}+\|\Bu\|_{{L}^{2}
					\left(\mathscr{B}_{1}
					(\Bx)\cap \Omega\right)}+\|\Bu\|_{L^{\infty}(\Omega)}\right)
				\leq C, \quad \Bx\in \partial\Omega.
			\end{equation}
			Furthermore,  due to the $H^{2}$ estimate \eqref{W1PESSTOKS2},
			one has
			\begin{equation}\label{W14ES}
				\|\nabla\Bu\|_{H^{1}\left(
					\mathscr{B}_{\frac{1}{2}}(\Bx)
					\cap \Omega\right)}\leq C\left(\|\Bu \cdot \nabla\Bu\|_{L^{2}\left(
					\mathscr{B}_{\frac{3}{4}}(\Bx)
					\cap \Omega\right)}
				+\|\Bu\|_{H^{1}\left(
					\mathscr{B}_{\frac{3}{4}}(\Bx)
					\cap \Omega\right)}\right)\leq C,
				\quad \Bx \in \partial\Omega.
			\end{equation}
			Finally, according to the $H^{3}$ estimate \eqref{W1PESSTOKS456}, one has
			\begin{equation}\label{W156ES}
				\|\nabla\Bu\|_{H^{2}\left(
					\mathscr{B}_{\frac{1}{4}}(\Bx)
					\cap \Omega\right)}\leq C\left(\|\Bu \cdot \nabla\Bu\|_{H^{1}\left(
					\mathscr{B}_{\frac{1}{2}}(\Bx)
					\cap \Omega\right)}+	\|\Bu\|_{H^{2}\left(
					\mathscr{B}_{\frac{1}{2}}(\Bx)
					\cap \Omega\right)}\right)\leq C,
				\quad \Bx \in \partial\Omega.
			\end{equation}
			\emph{Case 2. Interior regularity.}
			According to \cite[Theorem IV.4.1, Theorem IV.4.4, Remark IV.4.2]{GAGP11}, it holds that for any $\Bx\in \mathbb{R}^{2}\times \left[\frac{1}{8}, \frac{7}{8}\right]$,
			\begin{equation}
				\|\nabla\Bu\|_{L^{4}\left(
					\mathscr{B}_{\frac{5}{64}}(\Bx)\right)}\leq
				C\|\Bu\|^{2}_{L^{8}\left(
					\mathscr{B}_{\frac{7}{64}}(\Bx)\right)}+C\|\Bu\|_{L^{4}\left(
					\mathscr{B}_{\frac{7}{64}}(\Bx)\right)}\leq C.
			\end{equation}
			Moreover, one has
			\begin{equation}\label{W1peshh}
				\|\nabla\Bu\|_{W^{1,4}\left(\mathscr{B}_{\frac{1}{32}}(\Bx)\right)}
				\leq C\|\nabla\Bu \|_{L^{4}\left(\mathscr{B}_{\frac{5}{64}}(\Bx)\right)}\|\Bu\|_{L^{\infty}\left(\mathscr{B}_{\frac{5}{64}}(\Bx)\right)}+C\|\Bu \|_{W^{1,4}\left(\mathscr{B}_{\frac{5}{64}}(\Bx)\right)}\leq C.
			\end{equation}
			
			Hence it follows from the Sobolev embedding inequality and the estimates \eqref{W156ES}, \eqref{W1peshh} that $\nabla \Bu$ is bounded. The proof of Lemma  \ref{Le:NAvbouns} is completed.
		\end{proof}
	\end{appendices}	
	
	{\bf Acknowledgement.}
	The research was partially supported by the National Key R$\&$D Program of China, Project Number 2020YFA0712000.  The research of Wang was partially supported by NSFC grants 12171349 and 12271389. The research of  Xie was partially supported by  NSFC grants 11971307 and 1221101620, Fundamental Research Grants for Central  universities,  Natural Science Foundation of Shanghai 21ZR1433300, and Program of Shanghai Academic Research leader 22XD1421400.

\end{document}